\documentclass[a4paper,11pt,twoside]{amsart}

\usepackage{pgfplots}
\usepackage{mathrsfs}
\usetikzlibrary{intersections}
\usepackage{amsmath, amssymb, amsfonts, latexsym, mdwlist, amsthm, amscd}

\usepackage[left=1.0in,right=1.0in]{geometry}
\usepackage[all]{xy} 

\usepackage{subfig}
\usepackage{graphicx}
\usepackage{wrapfig}
\usepackage{graphicx}
\usepackage{amsthm}
\usepackage{tikz-cd}
\theoremstyle{definition}
\newtheorem{theorem}{Theorem}[section]
\newtheorem{definition}[theorem]{Definition}
\newtheorem{conjecture}[theorem]{Conjecture}

\newtheorem{lemma}[theorem]{Lemma}
\newtheorem{corollary}[theorem]{Corollary}
\newtheorem{proposition}[theorem]{Proposition}

\newtheorem{remark}[theorem]{Remark}

\usepackage[bookmarks, colorlinks, breaklinks, pdftitle={},
pdfauthor={}]{hyperref}
\hypersetup{linkcolor=blue,citecolor=blue,filecolor=black,urlcolor=blue}

\newcommand{\Addresses}{{
  \bigskip
  \footnotesize

  S.~Liu, \textsc{Mathematics Institute, University of Warwick, Coventry, CV4 7AL, United Kingdom}\par\nopagebreak
  \textit{E-mail address}: \texttt{Shengxuan.Liu.1@warwick.ac.uk}\par\nopagebreak
  \textit{URL:} \texttt{https://warwick.ac.uk/fac/sci/maths/people/staff/sliu/}

}}
\usepackage{tikz}
\usetikzlibrary{calc,cd,trees,positioning,arrows,chains,shapes.geometric,%
    decorations.pathreplacing,decorations.pathmorphing,shapes,%
    matrix,shapes.symbols}

\tikzset{
>=stealth',
  punktchain/.style={
    rectangle,
    rounded corners,
    draw=black, thick,
^{}    
    minimum height=3em,
    text centered,
    on chain},
  line/.style={draw, thick, <-},
  element/.style={
    tape,
    top color=white,
    bottom color=blue!50!black!60!,
    minimum width=8em,
    draw=blue!40!black!90, very thick,
    text width=10em,
    minimum height=3.5em,
    text centered,
    on chain},
  every join/.style={->, thick,shorten >=1pt},
  decoration={brace},
  tuborg/.style={decorate},
  tubnode/.style={midway, right=2pt},
}

\usepackage{paralist}
\setdefaultenum{(a)}{(i)}{}{}
\usepackage{enumitem} 

\usepackage{leftidx}
\usepackage{verbatim}

\def\cO{\mathcal{O}}

\def\ch{\mathop{\mathrm{ch}}\nolimits}

\def\Coh{\mathop{\mathrm{Coh}}\nolimits}

\def\deg{\mathop{\mathrm{deg}}}

\def\ext{\mathop{\mathrm{ext}}\nolimits}

\def\Hom{\mathop{\mathrm{Hom}}\nolimits}

\def\rk{\mathop{\mathrm{rk}}}

\def\MG13{\ensuremath{{\mathcal M}_{\Gamma_1(3)}}}
\def\tildeMG13{\ensuremath{\widetilde{\mathcal M}_{\Gamma_1(3)}}}

\def\abs#1{\left\lvert#1\right\rvert}

\makeatletter
\newtheorem*{rep@theorem}{\rep@title}
\newcommand{\newreptheorem}[2]{%
\newenvironment{rep#1}[1]{%
 \def\rep@title{#2 \ref{##1}}%
 \begin{rep@theorem}}%
 {\end{rep@theorem}}}
\makeatother

\newreptheorem{Thm}{Theorem}

\newreptheorem{Cor}{Corollary}

\newreptheorem{Con}{Conjecture}

\newtheorem{thm-int}{Theorem}

\theoremstyle{definition}
\newtheorem{Def-s}[Thm]{Definition}

\numberwithin{equation}{section}

\def\llra{\hbox to 10mm{\rightarrowfill}}
\def\lllra{\hbox to 15mm{\rightarrowfill}}

\def\llla{\hbox to 10mm{\leftarrowfill}}
\def\lllla{\hbox to 15mm{\leftarrowfill}}

\DeclareMathOperator{\Cone}{Cone}

\def\llra{\hbox to 10mm{\rightarrowfill}}
\def\lllra{\hbox to 15mm{\rightarrowfill}}

\newcommand{\ignore}[1]{}

\begin{document}

\title[]{Stability Condition on Calabi--Yau Threefold of Complete Intersection of Quadratic and Quartic Hypersurfaces}

\author[Shengxuan Liu]{Shengxuan Liu}



\begin{abstract}
In this paper, we prove a Clifford type inequality for the curve $X_{2,2,2,4}$, which is the intersection of a quartic and three general quadratics in $\mathbb{P}^5$. We thus prove a stronger Bogomolov--Gieseker inequality for characters of stable vector bundles and stable objects on Calabi--Yau complete intersection $X_{2,4}$. Applying the scheme proposed by Bayer, Bertram, Macr\`i, Stellari and Toda, we can construct an open subset of Bridgeland stability conditions on $X_{2,4}$.
\end{abstract}
\maketitle

\setcounter{tocdepth}{1}
\tableofcontents

\section{Introduction}
Stability conditions on a triangulated category were first introduced by Bridgeland in \cite{bridgeland2007stability} to understand $\Pi$-stability, which was proposed by Douglas. Since then, the existence of Bridgeland stability conditions on smooth projective varieties has become a central problem. In a series of works (\cite{bayer2011bridgeland,bayer2011bridgeland1,bayer2016space}), a conjectural Bogomolov--Gieseker inequality is proposed to construct geometric stability conditions on threefolds. The construction is based on the double tilting technique, and the existence of geometric stability conditions relies on the conjectural Bogomolov--Gieseker inequaltiy.

The existence of stability conditions on threefolds is only known for a limited number of examples: Fano threefolds \cite{macri2014generalized,schmidt2014generalized,li2018stability,piyaratne2017stability,bernardara2016bridgeland}, Abelian threefolds \cite{maciocia2015fourier,maciocia2016fourier,bayer2016space}, Kummer type threefolds \cite{bayer2016space}, threefolds with nef tangent bundle \cite{koseki2020nef}, and product varieties of a curve with a surface \cite{liu2021product}. On Calabi--Yau threefolds, apart from the above examples, Li proved the existence of stability conditions on quintic threefolds \cite{li2019stability}, and Koseki proved the existence on some weighted hypersurfaces in $\mathbb{P}_{1,1,1,1,2}$ and $\mathbb{P}_{1,1,1,1,4}$ \cite{koseki2020stability}. They both show the existence by proving that a stronger or modified Bogomolov--Gieseker inequality holds for these varieties. It is worth mentioning that the original conjectural Bogomolov--Gieseker inequality as in \cite{bayer2016space} does not hold for all threefolds. In particular, counterexamples can be found in \cite{Schmidt:CounterBG,martinez2019bridgeland}. However, a modified version of the Bogomolov--Gieseker inequality is sufficient to show the existence. An explanation of the construction can be found in \cite{bayer2016space,li2019stability,koseki2020stability}.

Let $X_{2,4}$ be a polarized smooth complete intersection of quadratic and quartic hypersurfaces with $H=O_{X_{2,4}}(1)$. The main result of this paper, which is the following theorem (\cite[Conjecture 4.1]{bayer2016space}) for $X_{2,4}\subset\mathbb{P}^5$, with mild restrictions on $\alpha,\beta$:

\begin{theorem}[Theorem \ref{thm:main3}]\label{thm:main2}
Conjecture 4.1 in \cite{bayer2016space} holds for $X_{2,4}$ when the parameters $(\alpha,\beta)$ satisfy $\alpha^2+(\beta-\lfloor\beta\rfloor-\frac{1}{2})^2>\frac{1}{4}$. More precisely, assume $E$ is $\nu_{\alpha,\beta,H}$-tilt semistable for some $\alpha^2+(\beta-\lfloor\beta\rfloor-\frac{1}{2})^2>\frac{1}{4}$, then the generalised Bogomolov--Gieseker inequality holds:
$$Q_{\alpha,\beta}(E):=\left(2\alpha-\beta^2\right)\overline{\Delta}_H(E)+4\left(H\ch_2^{\beta H}(E)\right)^2-6H^2\ch_1^{\beta H}(E)\ch_3^{\beta H}(E)\geq0.$$
\end{theorem}

By Theorem \ref{thm:main2} and the framework in \cite{bayer2014mmp,bayer2011bridgeland1,bayer2016space}, see \cite[Theorem 8.6, Proposition 8.10]{bayer2016space}, we have
\begin{theorem}\label{thm:main2'}
There is a continuous family $\sigma^{a,b}_{\alpha,\beta,H}=\left(Z_{\alpha,\beta,H}^{a,b}(X_{2,4}),\mathcal{A}^{\alpha,\beta,H}(X_{2,4})\right)$ of stability conditions on $X_{2,4}$, parameterized by $(\alpha,\beta,a,b)\in\mathbb{R}_{>0}\times\mathbb{R}\times \mathbb{R}_{>0}\times\mathbb{R}$ satisfying
$$\alpha^2+\left(\beta-\lfloor\beta\rfloor-\frac{1}{2}\right)^2>\frac{1}{4}\;\text{ and }\;a>\frac{\alpha^2}{6}+\frac{1}{2}|b|\alpha.$$
\end{theorem}
The detailed notation of the stability condition in the above theorem will be explained in Section 2.

The proof of Theorem \ref{thm:main2} is based on a Bogomolov--Gieseker type inequality on $X_{2,4}$, the smooth projective Calabi--Yau threefold of complete intersection of quadratic and quartic hypersurfaces:

\begin{proposition}[Proposition \ref{prop:main5}]\label{prop:main1}
Suppose $F$ is a slope semistable sheaf on $X_{2,4}$ with $\ch_0(F)\neq0$ then

\begin{equation}
    \frac{H\ch_2(F)}{H^3\ch_0(F)}\leq
    \begin{cases}
      \left(\frac{H^2\ch_1(F)}{H^3\ch_0(F)}\right)^2-\frac{H^2\ch_1(F)}{H^3\ch_0(F)} & \text{if } \frac{H^2\ch_1(F)}{H^3\ch_0(F)}\in[0,\frac{4}{3}-\frac{\sqrt{13}}{3}]\\
      \frac{5}{8}\left(\frac{H^2\ch_1(F)}{H^3\ch_0(F)}\right)^2-\frac{1}{8} & \text{if }\frac{H^2\ch_1(F)}{H^3\ch_0(F)}\in(\frac{4}{3}-\frac{\sqrt{13}}{3},\frac{1}{2}]\\
      \frac{5}{8}\left(\frac{H^2\ch_1(F)}{H^3\ch_0(F)}\right)^2-\frac{1}{4}\frac{H^2\ch_1(F)}{H^3\ch_0(F)} & \text{if } \frac{H^2\ch_1(F)}{H^3\ch_0(F)}\in(\frac{1}{2},\frac{\sqrt{13}}{3}-\frac{1}{3})\\
      \left(\frac{H^2\ch_1(F)}{H^3\ch_0(F)}\right)^2-\frac{1}{2} & \text{if }\frac{H^2\ch_1(F)}{H^3\ch_0(F)}\in[\frac{\sqrt{13}}{3}-\frac{1}{3},1)
    \end{cases}
\end{equation}
and for $\frac{H^2\ch_1(F)}{H^3\ch_0(F)}\in[n,n+1)$ for $n\in\mathbb{Z}$, we take $F(-nH)$ with the property $\frac{H^2\ch_1(F(-nH))}{H^3\ch_0(F(-nH))}\in[0,1)$. 
\end{proposition}

Recall that the classical Bogomolov inequality tells us that for a slope semistable sheaf $F$ on $X_{2,4}$, we have $\frac{H\ch_2(F)}{H^3\rk(F)}\leq \frac{1}2\left(\frac{H^2\ch_1(F)}{H^3\rk(F)}\right)^2$. From this we note that the above inequality is strictly stronger than the classical Bogomolov inequality unless $\frac{H^2\ch_1(F)}{H^3\rk(F)}$ takes integer values. 

On the other hand, as we prove the Bogomolov--Gieseker inequality for $X_{2,4}$, by applying the framework of \cite{feyzbakhsh2020curve,feyzbakhsh2021rank,Feyzbakhsh:2022,Feyzbakhsh}, we get explicit descriptions and properties of many enumerative invariants. In particular, let $v$ be a numerical K-theory class of rank $>0$ on $X_{2,4}$. Let $J(v)$ denote the Joyce--Song's generalized Donaldson--Thomas invariant counting the number of Gieseker semistable sheaves. Then by \cite{Feyzbakhsh:2022}, we have 
$$J(v)=F(J(\alpha_1),J(\alpha_2),...),$$
where $F$ is a universal polynomial in the invariants $J(\alpha_i)$ and the $\alpha_i$'s are rank $1$ characters. By \cite{Pand}, MNOP conjecture holds for $X_{2,4}$ and we can replace the $J(\alpha_i)$'s by corresponding Gromov--Witten invariants.

\subsection{Strategy of Proof}

In this subsection, we explain how we prove Theorem \ref{thm:main2}. The general logical flow will follow \cite{li2019stability,koseki2020stability} and thus we briefly review the proof of \cite{li2019stability}. Let $X_5$ be a smooth quintic threefold with polarization $H=O_{X_5}(1)$. Consider the following tower of smooth varieties:
$$C_{2,2,5}\subset S_{2,5}\subset X_5\;\;\text{and}\;\; C_{2,2,5}\subset S_{2,2},$$
where $C_{2,2,5}$, $S_{2,5}$, and $S_{2,2}$ are generic smooth complete intersections of $(2,2,5)$, $(2,5)$, and $(2,2)$ hypersurfaces, respectively. Then Theorem \ref{thm:main2} for $X_5$ is proved in following three steps:

\begin{enumerate}
    \item For a stable vector bundle $F$ on $C_{2,2,5}$, by pushing forward along the embedding $C_{2,2,5}\subset S_{2,2}$, one regards $F$ as a torsion sheaf on $S_{2,2}$. Then by wall-crossing technique on the Bridgeland stability space on $S_{2,2}$ as in \cite{feyzbakhsh2019mukai}, one gets a Clifford type inequality of $F$.
    \item Using Feyzbakhsh's restriction theorem in \cite{feyzbakhsh2016stability}, one recovers a Bogomolov--Gieseker type inequality on $S_{2,5}$. Using Feyzbakhsh's restriction theorem again, a Bogomolov--Gieseker type inequality on $X_5$ is obtained.
    \item By applying the Bogomolov--Gieseker type inequality on $X_5$ in step (b), one proves that $Q_{0,0}(E)\geq0$ for a Brill--Noether stable object $E$. Then by \cite[Theorem 5.4]{bayer2016space}, one gets Theorem \ref{thm:main2}.
\end{enumerate}

The main difference of this paper compared to \cite{li2019stability, koseki2020stability} is step (a), where we do not embed the curve $X_{2,2,2,4}$ into a del Pezzo surface, but a K3 surface of Picard rank $1$. Unlike the proof of \cite{li2019stability, koseki2020stability}, we do not have Hom vanishings of Brill--Noether semistable objects as in their cases of del Pezzo surfaces. Although a bound for global sections $\hom(O_S,E)$ is already given in \cite{feyzbakhsh2019mukai}, this bound is not strong enough to produce a Clifford type inequality to prove Theorem \ref{thm:main5}. Instead we bound $\hom(O_S,E)$ directly if $E$ is Brill--Noether semistable of negative slope, or $\ext^1(O_S,E)$ if $E$ is Brill--Noether semistable of positive slope and then use $\hom(O_S,E)=\chi(O_S,E)+\ext^1(O_S,E)$ to get the bound. The main technique we use is the semistability of the spherical twist of $E$ as in Proposition \ref{prop:boundsbnobj}. Next we use the convex polygon trick in \cite{feyzbakhsh2018higher} to get the Clifford type inequality, which is similar to \cite[Section 4]{li2019stability}.

\subsection{Plan of Paper}
The plan of this paper is the following. In Section 2, we recall basic notations, constructions and properties of stability conditions. In Section 3, with a similar method to \cite{li2019stability}, we prove Theorem \ref{thm:main2} by assuming that the results in Section 5, especially Theorem \ref{thm:main5}, hold. This corresponds to step (c) in Section 1.1. In Section 4, we deduce a Clifford type inequality for the curve  $C:=X_{2,2,2,4}$ for a slope semistable vector bundle $E$. The method follows the last part in Section 1.1. In Section 5, we deduce Theorem \ref{thm:main5} by using the Clifford type inequality in Section 4. The proof is similar to \cite{li2019stability}. This corresponds to step (b).

\section{Introduction to Bridgeland Stability}
\subsection{Notations of Slope Stability}
In this section, we review notations of slope stability conditions. Let $X$ be a smooth projective variety. The facts concerning slope stability can all be found in \cite{huybrechts2010geometry}.

\begin{definition}[Slope stability]
Let $F$ be a coherent sheaf on $X$ and let $H$ be an ample divisor. The slope of $F$ with respect to $H$ is
$$\mu_H(F)=
\begin{cases}
  \frac{H^{n-1}\ch_1(F)}{H^n\ch_0(F)} & \text{if }\ch_0(F)\neq0\\
  +\infty & \text{if } \ch_0(F)=0
\end{cases}
$$
A coherent sheaf $F$ is said to be slope (semi)stable if for any nontrivial coherent subsheaf $E\subset F$,
$$\mu_H(E)<(\leq)\mu_H(F/E).$$
\end{definition}

\begin{proposition}[Harder--Narasimhan Filtration]
For every coherent sheaf $F$, there is a filtration, called Harder--Narasimhan filtration, 
$$0=F_0\subset F_1\subset...\subset F_n=F$$
where $F_i/F_{i-1}$ is slope semistable and
$$\mu_H(F_1/F_0)>\mu_H(F_2/F_1)>...>\mu_H(F_n/F_{n-1}).$$
We denote $\mu^+(F)\;(\text{respectively}\;\mu^-(F))$ to be the maximum (respectively minimum) slope of the HN factors of $F$ under the slope stability function.
\end{proposition}

\subsection{Weak Stability Condition and Bridgeland Stability Condition}

Let $D^b(X)$ be the bounded derived category of coherent sheaves on $X$. From now on we consider that $X$ is of dimension $2$ or $3$.
\begin{definition}[Heart of a bounded t-structure]
A full additive subcategory $\mathcal{A}\subset D^b(X)$ is called the heart of a bounded t-structure if
\begin{enumerate}
    \item $\Hom(\mathcal{A}[i],\mathcal{A}[j])=0$ if $i>j$.
    \item For every $E\subset D^b(X)$, there is a chain of morphisms
    $$0=E_0\xrightarrow{\phi_0}E_1\xrightarrow{\phi_1}E_2\xrightarrow{\phi_2}...\xrightarrow{\phi_{n-1}}E_n=E$$
    such that $\Cone(\phi_i)\in\mathcal{A}[k_i]$ and $k_0>k_2>...>k_{n-1}$.
\end{enumerate}
\end{definition}
One can actually show that $\mathcal{A}$ is an abelian category.
\begin{definition}[Weak Stability Function]
Let $\mathcal{A}$ be an abelian category. A group homomorphism $Z:K_0(\mathcal{A})\rightarrow \mathbb{C}$ is called a weak stability function on $\mathcal{A}$ if, for $E\in\mathcal{A}$, we have $\Im(Z([E]))\geq 0$, with $\Im(Z([E]))=0\;\implies\;\Re(Z([E]))\leq 0$. The function $Z$ is called a stability function if, moreover, for all $0\neq E\in\mathcal{A}$, $\Im(Z([E]))=0\;\implies\;\Re(Z([E]))<0.$
\end{definition}

The original definition of Bridgeland stability is given by \cite{bridgeland2007stability}. By \cite[Proposition 5.3]{bridgeland2007stability}\cite[Lemma 5.11]{macri2017lectures}, a pair $\sigma=(\mathcal{A},Z)$, where $\mathcal{A}$ is the heart of a bounded t-structure, and the function $Z:K_0(\mathcal{A})\rightarrow\mathbb{C}$ is a group homomorphism, is a Bridgeland stability condition if it satisfies the following properties:
\begin{enumerate}
    \item[(1)]  The group homomorphism $Z$ is a stability function. We define the slope of a nonzero object $E$ in $\mathcal A$ to be $$\nu(E):= 
    \begin{cases}
  -\frac{\Re(Z([E]))}{\Im(Z([E]))} & \text{if }\Im(Z([E]))\neq0\\
  +\infty & otherwise.
\end{cases}$$ An object $E$ is called $\sigma$-(semi)stable if for any nonzero proper subobject $F\subset E$, we have $\nu(F)<(\leq)\nu(E/F)$.
    \item [(2)](HN filtrations) For any nonzero object $E\in\mathcal{A}$, we have a filtration of $E$, called the Harder--Narasimhan filtration:
$$0=E_0\subset E_1\subset...\subset E_m=E,$$
where $\nu(E_1/E_0)>\nu(E_2/E_1)>...>\nu(E_m/E_{m-1})$ and $E_i/E_{i-1}$ are semistable.

\item [(3)] (Support property) There is a constant $C>0$ such that for any semistable object $E$, we have $||E||\leq C\abs{Z([E])}$, where $||\cdot||$ is a fixed norm on $K_0(X)\otimes \mathbb{R}$.
\end{enumerate}
The pair $\sigma$ is called a weak stability condition if in (1), we only require $Z$ to be a weak stability function. Let $Stab(X)$ denote the set of all Bridgeland stability conditions on $X$, then Bridgeland showed that $Stab(X)$ is actually a complex manifold \cite[Proposition 8.1]{bridgeland2007stability}. Now we fix a class $v\in K_{num}(X)$, and consider the class of $\sigma$-semistable objects of class $v$ when $\sigma$ varies. Then we have a wall and chamber structure on $Stab(X)$ based on the following proposition:

\begin{proposition}[{\cite[Section 9]{bridgeland2008stability}\cite[Proposition 2.8]{Toda2008moduli}\cite[Proposition 3.3]{Bayer2011space}\cite[Proposition 2.3]{Bayer2014projectivity}}]
There exists a locally finite set $\mathscr{W}$ of real codimension $1$ submanifolds with boundary in $Stab(X)$ only depending on $v$, called walls, with the following properties:
\begin{enumerate}
    \item When $\sigma$ varies in a chamber (connected components of $Stab(X)\backslash (\cup_{W\in\mathscr{W}}W)$), then the sets of $\sigma$-semistable objects and $\sigma$-stable objects do not change.
    \item When $\sigma$ is on a single wall $W\in\mathscr{W}$, then there exists an object $F$ such that $F$ is unstable on the adjacent chamber of one side of the wall and semistable on the adjacent chamber on the other side of the wall.
\end{enumerate}

\end{proposition}

An important technique to construct weak or Bridgeland stability conditions is by tilting, which we recall now. Most of the materials here can be found in \cite{bayer2016space,macri2017lectures,li2019stability}.

We define the torsion pair $(\mathcal{T}_{\beta,H},\mathcal{F}_{\beta,H})$ as
$$\mathcal{T}_{\beta,H}=\left\{E\in \Coh(X)\middle|\mu_H^-(E)>\beta\right\}\;\;\text{and}\;\; \mathcal{F}_{\beta,H}=\left\{E\in\Coh(X)\middle|\mu_H^+(E)\leq\beta\right\}.$$
\begin{definition}
We let $\mathrm{Coh}^{\beta,H}(X)\subset D^b(X)$ be the extension-closure
$$\langle \mathcal{T}_{\beta,H},\mathcal{F}_{\beta,H}[1]\rangle.$$
\end{definition}

Then, by tilting, $\mathrm{Coh}^{\beta,H}(X)$ is the heart of a bounded t-structure, in particular an abelian category. Let $B\in NS(X)_\mathbb{R}$. We define the twisted  characters as
$$\ch^B_0(E)=\ch_0(E); \;\; \ch^B_1(E)=\ch_1(E)-B\ch_0(E);\;\;\ch_2^B(E)=\ch_2(E)-B\ch_1(E)+\frac{B^2\ch_0(E)}{2}$$
$$\ch^B_3(E)=\ch_3(E)-B\ch_2(E)+\frac{B^2\ch_1(E)}{2}-\frac{B^3\ch_0(E)}{6}$$

\begin{definition}[Tilt slope and stability]
Let $E\in \mathrm{Coh}^{\beta,H}(X)$, we define
$$\nu_{\alpha,\beta,H}(E)=
\begin{cases}
  \frac{H^{n-2}\ch_2(E)-\alpha H^n\ch_0(E)}{H^{n-1}\ch_1^{\beta H}(E)} & \text{if }H^{n-1}\ch_1^{\beta,H}(E)\neq0\\
  +\infty & \text{if }H^{n-1}\ch_1^{\beta H}(E)=0.
\end{cases}
$$
An object $E\in \mathrm{Coh}^{\beta,H}(X)$ is called $\nu_{\alpha,\beta,H}$-tilt (semi)stable if for any nontrivial subobject $F\subset E$ inside $\mathrm{Coh}^{\beta,H}(X)$, we have
$$\nu_{\alpha,\beta,H}(F)<(\leq)\nu_{\alpha,\beta,H}(E/F).$$
An object $E\in D^b(X)$ is called $\nu_{\alpha,\beta,H}$-tilt (semi)stable if $E[n]\in \mathrm{Coh}^{\beta,H}(X)$ is $\nu_{\alpha,\beta,H}$-tilt (semi)stable for some integer $n$.
\end{definition}
Same as slope stability, $\nu_{\alpha,\beta,H}$-stability also admits Harder--Narasimhan filtration property when $\alpha>\frac{\beta^2}{2}$. For an object $E\in \mathrm{Coh}^{\beta,H}(X)$, we denote $\nu^+_{\alpha,\beta,H}(E)$ (respectively $\nu^-_{\alpha,\beta,H}(E)$) as the maximum (respectively minimum) slopes of $\nu_{\alpha,\beta,H}$-HN filtration factors. We also write the central charge
$$Z_{\alpha,\beta,H}(E):=-\left(H^{n-2}\ch_2(E)-\alpha H^n\ch_0(E)\right)+iH^{n-1}\ch_1^{\beta H}(E).$$
\begin{definition}
Let $E\in \Coh^{\beta,H}(X)$. We define the $H$-discriminant of $E$ to be 
$$\overline{\Delta}_H(E):=\left(H^{n-1}\ch_1(E)\right)^2-2H^n\ch_0(E)\cdot H^{n-2}\ch_2(E).$$
\end{definition}
\begin{theorem}[\cite{bogomolov1978holomorphic},{\cite[Theorem 7.3.1]{bayer2011bridgeland}},{\cite[Proposition 2.21]{piyaratne2019moduli}}]\label{thm:bogomolov}
Let $E$ be a $\nu_{\alpha,\beta,H}$-semistable object for $\alpha>\frac{1}{2}\beta^2$. Then $\overline{\Delta}_H(E)\geq0.$
\end{theorem}

\subsection{Bridgeland Stability on K3 surfaces}

In this section we recall the Bridgeland stability condition on K3 surfaces. Most of this can be found in \cite{bridgeland2008stability}. A review of it can be found in \cite{feyzbakhsh2018higher}.

Let $(S,H)$ be a polarized K3 surface with Picard group generated by $H$. We use the same heart of bounded t-structure $\mathrm{Coh}^{\beta,H}(S)$ as above,
and define the central charge as
$$Z_{\alpha,\beta}(E)=-\ch_2(E)+\alpha \rk(E)+i\left(\frac{H\ch_1(E)}{H^2}-\beta \rk(E)\right).$$
Note that this is slightly different from the central charge given above, but they give the same family of weak stability conditions. We will only use this family of stability conditions on K3 surfaces.

Now we define a periodic function of period 1,  
$$\gamma(x)=
\begin{cases}
  1-x^2 & \text{if } x\in[-\frac{1}{2},\frac{1}{2}]\backslash \{0\}\\
  0 & \text{if } x=0.
\end{cases}$$ 
Let $\Gamma(x)=\frac{H^2}{2}x^2-\gamma(x)$, and $\Gamma_+$ to be the region above $\Gamma$. Then Bridgeland showed that $(\alpha,\beta)$ with $\alpha>\Gamma(\beta)$ defines a stability condition $\sigma_{\alpha,\beta}=(\mathrm{Coh}^{\beta,H}(X),Z_{\alpha,\beta})$ on K3 surfaces (\cite{bridgeland2008stability,feyzbakhsh2018higher}) by using Theorem \ref{thm:bogomolov}. The slope is defined to be  $-\frac{\Re(Z_{\alpha,\beta})}{\Im(Z_{\alpha,\beta})}$. If $E$ is an object with $\rk(E)\neq0$, we define 
$$pr(E):=pr(\ch(E))=\left(\frac{H\ch_1(E)}{H^2\rk(E)},\frac{\ch_2(E)}{\rk(E)}\right).$$

By abuse of notation, we also write $\Gamma$ for the graph on the plane.
Then we have a description of walls:

\begin{proposition}[{\cite[Proposition 9.3]{bridgeland2008stability}}]\label{prop:k3w}
Let $F\in D^b(X)$. Then there exists a collection of line segments (walls) $\mathcal{W}^i_F$ in $\Gamma_+$ with the following property:
\begin{enumerate}
    \item[(1)] The end points are either on $\Gamma$, or on the vertical line segment $(n,\frac{H^2}{2}n^2)$ to $\left(n,\frac{H^2}{2}n^2-1\right)$;

\item[(2)] The extension of the wall passes through $pr(F)$ if $\rk(F)\neq0$, otherwise it has slope $\frac{H^2\ch_2(F)}{H\ch_1(F)}$;

\item[(3)] The stability of $F$ does not change between two consecutive walls;

\item[(4)]  $F$ is strictly $\sigma_{\alpha,\beta}$-semistable (in other words semistable but not stable) if $(\alpha,\beta)$ is contained in one of the walls;

\item[(5)] If $F$ is semistable on one side of the wall, then it is unstable on the other side of the wall.
\end{enumerate}
\end{proposition}

The above proposition also holds if we consider a numerical class $v$. Then the proposition holds for potential walls, which means the walls that can happen for an object $F$ with $v(F)=v$. The following proposition is important for us to give an upper bound for the global sections of Brill--Noether semistable objects.

\begin{proposition}[{\cite[Theorem 2.15(a)]{bayer2014mmp}}]\label{prop:gamma}
Let $Stab^\dagger(X)$ denote the connected components of $Stab(X)$ containing geometric stability conditions. In particular, the stability conditions $\sigma_{\alpha,\beta}$ given above are in $Stab^\dagger(X)$. Let $v=mv_0\in H^*_{alg}(X;\mathbb{Z})$ be a Mukai vector with $v_0$ primitive and $m > 0$, and let
$\sigma \in Stab^\dagger(X)$ be a generic stability condition with respect to $v$. (This means that $\sigma$ is not on the wall of $v$.) Then the coarse moduli space $M_\sigma(v)$ is non-empty if and only if $v_0^2\geq -2$.
\end{proposition}

Now we briefly explain how to relate $v_0^2\geq-2$ with the curve $\Gamma$. Let $v_0=(\rk,cH,s)$. If $cH=0$, then $-2\rk \ch_2-2{\rk}^2\geq-2$. This implies $\frac{\ch_2}{\rk}\leq\frac{1}{{\rk}^2}-1\leq0$. Now suppose $c\neq 0$. Suppose $\frac{H\ch_1}{H^2\rk}\in[-\frac{1}{2},\frac{1}{2}]$, then we have 
$$-2\rk \ch_2+(\ch_1)^2-2{\rk}^2\geq-2\geq-2c^2.$$
This implies
$$\frac{\ch_2}{\rk}\leq\frac{H^2}{2}\left(\frac{H\ch_1}{H^2\rk}\right)^2-1+\left(\frac{H\ch_1}{H^2\rk}\right)^2.$$
If $\frac{H\ch_1}{H^2\rk}\in[n-\frac{1}{2},n+\frac{1}{2}]\backslash \{n\}$, we have
$$-2\rk \ch_2+(\ch_1)^2-2{\rk}^2\geq-2\geq-2(c-n\rk)^2.$$
And thus we get the curve $\Gamma$. So we see that if an object is $\sigma$-semistable for some generic $\sigma$, we have $\frac{\ch_2}{\rk}\leq\Gamma\left(\frac{H\ch_1}{H^2\rk}\right).$ For a more detailed explaination of $\Gamma$, we would like to ask readers to consult \cite{lahoz2021chern,fu2021stability} on Le Potier functions.

\subsection{Stability Conditions on $X_{2,4}$}

The goal of this paper is to show that the following conjectural Bogomolov--Gieseker inequality holds when $X$ is the complete intersection of quadratic and quartic hypersurfaces in $\mathbb{P}^5$:

\begin{conjecture}[{\cite[Conjecture 2.7]{bayer2011bridgeland}},{\cite[Conjecture 4.1]{bayer2016space}}]\label{coj}
Let $X$ be a smooth projective threefold of complete intersection of quadratic and quartic hypersurfaces in $\mathbb{P}^5$, and let $H$ be an ample class. Assume $E$ is $\nu_{\alpha,\beta,H}$-tilt semistable for some $\alpha>\frac{1}{2}\beta^2$. Then
$$Q_{\alpha,\beta}(E):=\left(2\alpha-\beta^2\right)\overline{\Delta}_H(E)+4\left(H\ch_2^{\beta H}(E)\right)^2-6H^2\ch_1^{\beta H}(E)\ch_3^{\beta H}(E)\geq0.$$
\end{conjecture}

\begin{remark}
It was expected that Conjecture \ref{coj} is correct for all projective threefolds. However, this is not the case. A counterexample is given in \cite{Schmidt:CounterBG}: The above inequality is violated in blowing up a point on $\mathbb{P}^3$. A weaker Bogomolov--Gieseker inequality is proposed and proved for all Fano threefolds, see \cite{bernardara2016bridgeland}. 
\end{remark}

In this paper, we prove the above conjecture with small restrictions on $\alpha,\beta$:
\begin{theorem}\label{thm:main3}
Let $X$ be a smooth projective threefold of complete intersection of quadratic and quartic hypersurfaces in $\mathbb{P}^5$, and let $H=O_X(1)$. Assume that $E$ is $\nu_{\alpha,\beta,H}$-tilt semistable for some $\alpha>\frac{1}{2}\beta^2+\frac{1}{2}(\beta-\lfloor\beta\rfloor)(\lfloor\beta\rfloor+1-\beta)$. Then $Q_{\alpha,\beta}(E)\geq0.$
\end{theorem}

Under the framework of \cite{bayer2011bridgeland,bayer2011bridgeland1,bayer2016space}, we can construct a family of Bridgeland stability conditions on $X_{2,4}$. The heart $\mathcal{A}$ of the stability condition is constructed by ``double-tilting" $\Coh(X_{2,4})$. We define the double tilting heart $\mathcal{A}^{\alpha,\beta,H}$ to be the extension $\langle\mathcal{T}_{\alpha,\beta,H},\mathcal{F}_{\alpha,\beta,H}[1]\rangle$, where
$$\mathcal{T}_{\alpha,\beta,H}=\left\{E\in \mathrm{Coh}^{\beta,H}(X)\middle|\nu_{\alpha,\beta,H}^{-}(E)>0\right\}\;\;\text{and}\;\;\mathcal{F}_{\alpha,\beta,H}=\left\{E\in \mathrm{Coh}^{\beta,H}(X)\middle|\nu_{\alpha,\beta,H}^{+}(E)\leq0\right\}.$$

We define the central charge $Z$ on $\mathcal{A}^{\alpha,\beta,H}$ to be
$$Z^{a,b}_{\alpha,\beta,H}(E):=\left(-\ch_3^{\beta H}(E)+bH\ch_2^{\beta H}(E)+aH^2\ch_1^{\beta H}(E)\right)+i\left(H\ch_2^{\beta H}(E)-\frac{\alpha^2}{2}H^3\rk(E)\right).$$

By Theorem \ref{thm:main2} and the framework in \cite{bayer2014mmp,bayer2011bridgeland1,bayer2016space}, see \cite[Theorem 8.6, Proposition 8.10]{bayer2016space}, we have
\begin{theorem}
There is a continuous family $\sigma^{a,b}_{\alpha,\beta,H}=\left(Z_{\alpha,\beta,H}^{a,b}(X_{2,4}),\mathcal{A}^{\alpha,\beta,H}(X_{2,4})\right)$ of stability conditions on $X_{2,4}$, parameterized by $(\alpha,\beta,a,b)\in\mathbb{R}_{>0}\times\mathbb{R}\times \mathbb{R}_{>0}\times\mathbb{R}$ satisfying
$$\alpha^2+\left(\beta-\lfloor\beta\rfloor-\frac{1}{2}\right)^2>\frac{1}{4}\;\text{ and }\;a>\frac{\alpha^2}{6}+\frac{1}{2}|b|\alpha.$$
\end{theorem}

\subsection{Useful Lemmas}

Now let 
$$\overline{v}_H(E)=\left(H^n\ch_0(E),H^{n-1}\ch_1(E),H^{n-2}\ch_2(E)\right)$$
$$\text{and }\;p_H(E)=\left(\frac{H^{n-2}\ch_2(E)}{H^n\ch_0(E)},\frac{H^{n-1}\ch_1(E)}{H^n\ch_0(E)}\right).$$ 
Then we have the following lemma on stability conditions:

\begin{lemma}\label{lem:wall}
Let $E$ be a $\nu_{\alpha_0,\beta_0,H}$-tilt stable object in $\Coh^{\beta_0,H}(X)$ for some $\alpha_0>\frac{1}{2}\beta_0$. Then
\begin{enumerate}
    \item {\cite[Corollary 3.3.3]{bayer2011bridgeland}}{\cite[Appendix B]{bayer2016space}} There is an open neighbourhood $U$ of $(\alpha_0,\beta_0)$ such that $E$ is $\nu_{\alpha,\beta,H}$-tilt stable for all $(\alpha,\beta)\in U$ with $\alpha>\frac{1}{2}\beta^2$.

\item {\cite[Theorem 3.1]{macri2014generalized}}{\cite[Lemma 4.3]{bayer2016space}}{\cite[Lemma 2.9]{li2019stability}}(Bertram's Nested Wall Theorem) The object $E$ is $\nu_{\alpha,\beta,H}$-tilt stable for any $(\alpha,\beta)$ with $\alpha>\frac{1}{2}\beta^2$ on the line through $(\alpha_0,\beta_0)$ and $p_H(E)$. More precisely, the object $E$ is $\nu_{\alpha,\beta,H}$-tilt stable for $(\alpha,\beta)$ such that
$$\det
\begin{pmatrix}
1 & \alpha & \beta\\
1 & \alpha_0 & \beta_0\\
H^n\ch_0(E) & H^{n-2}\ch_2(E) & H^{n-1}\ch_1(E)
\end{pmatrix}=0$$
The statement also holds for the semistable case. Moreover, when $X$ is a threefold,
$$H^{n-1}\ch_1^{\beta H}(E)Q_{\alpha_0,\beta_0}(E)=H^{n-1}\ch_1^{\beta_0 H}(E)Q_{\alpha,\beta}(E).$$

\item[(b')] Let $F$ be an object in $\Coh^{\beta,H}(X)$ such that $p_H(F)$ is on the line through the points $(\alpha_0,\beta_0)$ and $p_H(E)$, then $\nu_{\alpha_0,\beta_0,H}(E)=\nu_{\alpha_0,\beta_0,H}(F)$. More precisely, the requirements on $E$ and $F$ are as follows: Both $\overline{v}_H(E)$ and $\overline{v}_H(F)$ are not zero and  
$$\det
\begin{pmatrix}
1 & \alpha_0 & \beta_0\\
H^n\ch_0(E) & H^{n-2}\ch_2(E) & H^{n-1}\ch_1(E)\\
H^n\ch_0(F) & H^{n-2}\ch_2(F) & H^{n-1}\ch_1(F)
\end{pmatrix}=0$$

\item {\cite[Corollary 3.3.3]{bayer2011bridgeland}}{\cite[Appendix B]{bayer2016space}} (Destabilising walls) The set
$$\left\{(\alpha,\beta)\middle|\alpha>\frac{1}{2}\beta^2\text{ and }E\text{ is strictly }\nu_{\alpha,\beta,H}\text{-semistable}\right\}$$
is either empty or a union of lines and rays.
\end{enumerate}
\end{lemma}

The following lemma is essential to the proof in the following sections.

\def\dbar{\overline{\Delta}_H}

\begin{lemma}[{\cite[Corollary 3.10]{bayer2016space}}]\label{lem:dis}
Let $E$ be a strictly $\nu_{\alpha,\beta,H}$-tilt semistable object with finite slope. Then, for any of the Jordan--H\"older factors $E_i$ of $E$, we have
$$\overline{\Delta}_H(E_i)\leq\overline{\Delta}_H(E).$$
The equality only holds when $\overline{\nu}_H(E_i)=\overline{\nu}_H(E)$ and $\overline{\Delta}_H(E)=\overline{\Delta}_H(E_i)=0$.
\end{lemma}

One more (weak) stability condition we use is the Brill--Noether stability condition. A detailed explanation can be found in \cite{bayer2015wall}.

\begin{definition}
An object $E$ is called Brill--Noether stable if there is an open subset
$$U_{\delta}=\{(\alpha,\beta)|\alpha^2+\beta^2<\delta\text{ and }\alpha>\frac{1}{2}\beta^2\},$$
such that $E$ is $\nu_{\alpha,\beta,H}$-tilt stable for all $(\alpha,\beta)\in U_{\delta}$. 

We call an object $E$ Brill--Noether semistable if there exists $\delta>0$ such that $E$ is $\nu_{\alpha,0,H}$-semistable for every $0<\alpha<\delta$.

The Brill--Noether slope is defined by $\nu_{BN}(E)=\frac{H^{n-2}\ch_2(E)}{H^{n-1}\ch_1(E)}$ if $H^{n-1}\ch_1(E)\neq0$, and $+\infty$ otherwise.
\end{definition}

By Lemma \ref{lem:wall}, an object $E$ with $H^{n-2}\ch_2(E)\neq0$ is Brill--Noether stable if and only if it is $\nu_{\alpha,\beta,H}$-tilt stable for some $(\alpha,\beta)$ proportional to $p_H(E)$, and a similar result holds for the Brill--Noether semistable case.

The following well-known lemma will be used, and the proof can be found in \cite[lemma 2.12]{li2019stability}.

\begin{lemma}[{\cite[Lemma 6.5]{bayer2015wall}}]\label{lem:twist}
Assume that $E\in\Coh^{0,H}(X)$ is Brill--Noether stable. If $\nu_{BN}(E)>0$, let $W\subset \Hom(O_X,E)$ be a vector subspace and let
$$\Tilde{E}=\mathrm{Cone}(W\otimes O_X\xrightarrow{\mathop{ev}} E),$$
where the map is the evaluation map. Then the spherical twist $\Tilde{E}$ is also Brill--Noether semistable and $\Tilde{E}\in \mathrm{Coh}^{0,H}(X)$.

If $\nu_{BN}(E)<0$, let $W'\subset \Hom(E[-1],O_X)$ be a vector subspace and let
$$\Tilde{E}'=\mathrm{Cone}(E[-1]\xrightarrow{\mathop{can}} W'\otimes O_X),$$
where the map is the canonical map. Then the spherical twist $\Tilde{E}'$ is also Brill--Noether semistable and $\Tilde{E}'\in \mathrm{Coh}^{0,H}(X)$.
\end{lemma}

\section{Proof of the Main Result}

Let $X:=X_{2,4}\subset\mathbb{P}^5$ be a smooth complete intersection of quadratic and quartic hypersurfaces. Thus it is a Calabi--Yau $3$-fold by the adjunction formula. By Hirzebruch--Riemann--Roch, we have
$$\chi(E)=\frac{7}{12}H^2\ch_1(E)+\ch_3(E).$$

By the same argument as that in \cite[Theorem 3.2]{li2019stability}, we get the following theorem and proposition:

\begin{theorem}[{\cite[Theorem 3.2]{li2019stability}},{\cite[Theorem 5.4]{bayer2016space}}]
Theorem \ref{thm:main2} holds if 
$$Q_{0,0}(E):=4(H\ch_2(E))^2-6\left(H^2\ch_1(E)\right)(\ch_3(E))\geq0$$ 
for any Brill--Noether stable object $E\in \mathrm{Coh}^{0,H}(X)$ with $\nu_{BN}(E)\in[-\frac{1}{2},\frac{1}{2}]$.
\end{theorem}

\begin{proposition}\label{prop:main3}
The inequality $Q_{0,0}(E)\geq0$ holds for any Brill--Noether stable object $E\in \mathrm{Coh}^{0,H}(X)$ with $\nu_{BN}(E)\in[-\frac{1}{2},\frac{1}{2}]$.
\end{proposition}

The proof of this proposition depends on the positivity of the slope $\nu_{BN}(E)$. First we consider $\nu_{BN}(E)\in(0,\frac{1}{2}]$.

\begin{lemma}\label{lemma:hombound}
Let $E\in \mathrm{Coh}^{0,H}(X)$ be a Brill--Noether stable object with $\nu_{BN}(E)\in(0,\frac{1}{2}]$. Then we have
$$Q_{0,0}(E)\geq 4(H\ch_2(E))^2-6(H^2\ch_1(E))\rk(E)+\frac{7}{4}(H^2\ch_1(E))^2-4(H\ch_2(E))(H^2\ch_1(E)).$$
\end{lemma}

\begin{proof}
Let $E\in \mathrm{Coh}^{0,H}(X)$ be a Brill--Noether stable object with $\nu_{BN}(E)\in(0,\frac{1}{2}]$. There exists $(\alpha,\beta)$ such that $\alpha>\frac{1}{2}\beta^2$, $0<\frac{\alpha}{\beta}<\nu_{BN}(E)$ and $E$ is $\nu_{\alpha,\beta,H}$-tilt stable. Note that 
$$\nu_{\alpha,\beta,H}(O_X[1])=\frac{\alpha}{\beta}<\nu_{\alpha,\beta,H}(E).$$

By tilt stability and Serre duality, we have
\begin{equation}\label{eq1}
    \Hom(O_X,E[2+i])=(\Hom(E,O_X[1-i]))^*=0
\end{equation}
for $i\geq0$. Consider $\Tilde{E}:=\mathrm{Cone}(O_X\otimes \Hom(O_X,E)\rightarrow E)$. By Lemma \ref{lem:twist}, $\Tilde{E}$ is also Brill--Noether semistable. By Theorem \ref{thm:main5}, the slope $\frac{H^2\ch_1(\tilde{E})}{H^3\rk(\tilde{E})}$ cannot be in $(-\frac{1}{5},0]$. Then either
$$\frac{H^2\ch_1(E)}{H^3(\mathrm{rk}(E)-\mathrm{hom}(O_X,E))}=\frac{H^2\ch_1(\tilde{E})}{H^3\mathrm{rk}(\tilde{E})}\notin[-\frac{1}{2},-\frac{1}{5}],$$
which implies
$$\mathrm{hom}(O_X,E)<\mathrm{rk}(E)+\frac{1}{4}H^2\ch_1(E),$$
or $$\frac{H^2\ch_1(\tilde{E})}{H^3\mathrm{rk}(\tilde{E})}\in[-\frac{1}{2},-\frac{1}{5}],$$
which implies 
$$\mathrm{hom}(O_X,E)\leq \mathrm{rk}(E)+\frac{7}{24}H^2\ch_1(E)+\frac{2}{3}H\ch_2(E).$$

As we require $E\in \Coh^{0,H}(E)$ and $\nu_{BN}(E)\in(0,\frac{1}{2}]$, we know that $H\ch_2(E)>0$ and $H^2\ch_1(E)>0$ and thus in both cases we always have 
$$\mathrm{hom}(O_X,E)\leq \mathrm{rk}(E)+\frac{7}{24}H^2\ch_1(E)+\frac{2}{3}H\ch_2(E).$$
By slope we have $\Hom(O_X[1],E[j])=0$ for $j\leq -1$, and combining this with \eqref{eq1} we have
$$\chi(O_X,E)\leq \hom(O_X,E).$$
By Hirzebruch--Riemann--Roch, we have
$$\frac{7}{12}H^2\ch_1(E)+\ch_3(E)\leq \rk(E)+\frac{7}{24}H^2\ch_1(E)+\frac{2}{3}H\ch_2(E).$$
By multiplying with $6H^2\ch_1(E)$ and rearranging terms, we have
$$Q_{0,0}(E)\geq 4(H\ch_2(E))^2-6(H^2\ch_1(E))\rk(E)+\frac{7}{4}(H^2\ch_1(E))^2-4(H\ch_2(E))(H^2\ch_1(E)).$$

\end{proof}

\begin{lemma}\label{lemma:pos}
For a Brill--Noether stable object $E\in \Coh^{0,H}(X)$ with $\nu_{BN}(E)\in(0,\frac{1}{2}]$, we have $Q_{0,0}(E)\geq0$.
\end{lemma}

\begin{proof}
Let $E\in \mathrm{Coh}^{0,H}(X)$ be a Brill--Noether stable object with $\nu_{BN}(E)\in(0,\frac{1}{2}]$. By Lemma \ref{lemma:hombound}, we have the following inequality on $Q_{0,0}$:

\begin{align}
    Q_{0,0}(E) & \geq 4(H\ch_2(E))^2-6(H^2\ch_1(E))\rk(E)+\frac{7}{4}(H^2\ch_1(E))^2-4(H\ch_2(E))(H^2\ch_1(E)) \notag \\
    &  = 4\left((H\ch_2(E))-\frac{1}{2}(H^2\ch_1(E))\right)^2+\frac{3}{4}\left(H^2\ch_1(E)\right)^2-\frac{3}{4}(H^2\ch_1(E))(H^3\rk(E)).\label{eq2}
\end{align}
As we assume $\nu_{BN}(E)\in(0,\frac{1}{2}]$, we have $\frac{H^2\ch_1(E)}{H^3\rk(E)}\notin[0,\frac{3}{7}]$ by Theorem \ref{thm:main5}. 

\textbf{Case 1} When $\frac{H^2\ch_1(E)}{H^3\rk(E)}\notin[\frac{3}{7},1]$, then we have $H^2\ch_1(E)>H^3\rk(E)$ and thus by \eqref{eq2}, we have $Q_{0,0}(E)\geq0$.

\textbf{Case 2} When $\frac{H^2\ch_1(E)}{H^3\rk(E)}\in[\frac{4}{5},\frac{10}{11}]$, by Theorem \ref{thm:main5}, we have 
$$-\rk(E)\geq \frac{11}{64}H\ch_2(E)-\frac{51}{256}H^2\ch_1(E).$$
Thus
$$Q_{0,0}(E)\geq 4(H\ch_2(E))^2-\frac{95}{32}(H\ch_2(E))(H^2\ch_1(E))+\frac{71}{128}(H^2\ch_1(E))^2\geq0.$$

\textbf{Case 3} When $\frac{H^2\ch_1(E)}{H^3\rk(E)}\in[\frac{1}{2},\frac{4}{5}]$, by Theorem \ref{thm:main5}, we have
$$-6\rk(E)\geq3H\ch_2(E)-\frac{27}{16}H^2\ch_1(E)$$
and thus
$$Q_{0,0}(E)\geq 4\left((H\ch_2(E)-\frac{1}{8}(H^2\ch_1(E))\right)^2\geq0.$$

\textbf{Case 4} When $\frac{H^2\ch_1(E)}{H^3\rk(E)}\in [\frac{3}{7},\frac{1}{2}]$, by Theorem \ref{thm:main5}, we have
$$-6\rk(E)\geq 4H\ch_2(E)-\frac{7}{4}H^2\ch_1(E)$$
and thus we have
$$Q_{0,0}(E)\geq4(H\ch_2(E))^2\geq0.$$

\textbf{Case 5} When $\frac{H^2\ch_1(E)}{H^3\rk(E)}\in[\frac{10}{11},1]$, we need a better bound of $Q_{0,0}$. There are two subcases: 

\textbf{(1)} $\frac{H\ch_2(E)}{H^2\ch_1(E)}\leq\frac{79}{220}$. In this case, we have
\begin{align}
Q_{0,0}(E) & \geq 
 4(H\ch_2(E))^2-4(H\ch_2(E))(H^2\ch_1(E))+\frac{7}{4}(H^2\ch_1(E))^2-\frac{3}{4}(H^2\ch_1(E))(H^3\rk(E))\notag \\
&\geq\frac{2509}{3025}(H\ch_1(E))^2-\frac{3}{4}(H^2\ch_1(E))(H^3\rk(E))\notag\\
&\geq\left(\frac{2509}{3025}\times\frac{10}{11}-\frac{3}{4}\right)(H^2\ch_1(E))(H^3\rk(E))>0.\notag
\end{align}
where the second inequality follows from considering $f(x)=4x^2-4x+\frac{7}{4}$ for $x\leq\frac{79}{220}$, and the third inequality follows from $\frac{H^2\ch_1(E)}{H^3\rk(E)}\geq\frac{10}{11}$.

\textbf{(2)} $\frac{H\ch_2(E)}{H^2\ch_1(E)}\geq\frac{79}{220}$. In this case we consider the line $y=\frac{79}{220}x$ and the parabola $y=\frac{5}{8}x^2-\frac{1}{8}$. They intersect at $x=\frac{79-3\sqrt{2374}}{275}>-\frac{1}{4}$. Then we know that in this region, we always have 
$$\frac{H\ch_2(\tilde{E})}{H^3\rk(\tilde{E})}\leq-\frac{9}{32}\frac{H^2\ch_1(\tilde{E})}{H^3\rk(\tilde{E})}-\frac{5}{32}.$$
By a similar calculation as above, we have
$$\hom(O_X,E)\leq \rk(E)+\frac{9}{40}H^2\ch_1(E)+\frac{4}{5}H\ch_2(E),$$
$$\ch_3(E)\leq \rk(E)-\frac{43}{120}H^2\ch_1(E)+\frac{4}{5}H\ch_2(E),$$
and thus we have 
\begin{align}
Q_{0,0}&\geq4(H\ch_2(E))^2-6(H^2\ch_1(E))(\rk(E))+\frac{43}{20}(H^2\ch_1(E))^2-\frac{24}{5}(H\ch_2(E))(H^2\ch_1(E))\notag\\
&=\frac{4}{5}(\frac{3}{2}(H^2\ch_1(E))^2-(H\ch_2(E))(H^3\rk(E))-(H^2\ch_1(E))(H^3\rk(E)))\label{c1}\\
&+\frac{9}{20}(H^2\ch_1(E))(H^3\rk(E)-H^2\ch_1(E))\label{c2}\\
&+\frac{1}{5}(7H^2\ch_1(E)-10H\ch_2(E)-2H^3\rk(E))(H^2\ch_1(E)-2H\ch_2(E)).\label{c3}
\end{align}
The first term \eqref{c1} is nonnegative, since by Theorem \ref{thm:main5} we have 
$$(H^2\ch_1(E))^2-(H\ch_2(E))(H^3\rk(E))-\frac{1}{2}(H^3\rk(E))^2\geq 0$$
and
\begin{align*}
&\frac{3}{2}(H^2\ch_1(E))^2-(H\ch_2(E))(H^3\rk(E))-(H^2\ch_1(E))(H^3\rk(E))\\
&-(H^2\ch_1(E))^2+(H\ch_2(E))(H^3\rk(E))+\frac{1}{2}(H^3\rk(E))^2\\
&=\frac{1}{2}(H^2\ch_1(E)-H^3\rk(E))^2\geq0.
\end{align*}

The second term \eqref{c2} is also nonnegative because $\frac{H^2\ch_1(E)}{H^3\rk(E)}\leq1$. The third term \eqref{c3} is also nonnegative, because by Theorem \ref{thm:main5} we have $H\ch_2(E)\leq\frac{21}{11}(H^2\ch_1(E))-\frac{31}{22}(H^3\rk(E))$ and thus we have 
$$7H^2\ch_1(E)-10H\ch_2(E)-2H^3\rk(E)\geq\frac{133}{11}(H^3\rk(E)-H^2\ch_1(E))\geq 0.$$
Also $H^2\ch_1(E)-2H\ch_2(E)\geq0$, since we assumed $\frac{H\ch_2(E)}{H^2\ch_1(E)}\leq\frac{1}{2}$ from the beginning.

Thus, when $\nu_{BN}(E)\in(0,\frac{1}{2}]$, we have $Q_{0,0}(E)\geq0$.

\end{proof}

Next we show the case $\nu_{BN}(E)\in[-\frac{1}{2},0)$. A direct proof like Lemma \ref{lemma:pos} is possible, but we follow an enlightening method in \cite[Proposition 3.3]{li2019stability} by considering the derived dual $\mathbb{D}(E):=E^*[1]$ to reduce to Lemma \ref{lemma:pos}. First we recall the following proposition:

\begin{proposition}[{\cite[Proposition 5.1.3(b)]{bayer2011bridgeland}}]\label{prop:dual}
The derived dual $\mathbb{D}(E)$ fits into a dinstinguished triangle:
\begin{equation}
\overline{E}\rightarrow\mathbb{D}(E)\rightarrow T_0[-1]\rightarrow \overline{E}[1]   
\end{equation}
where $\overline{E}$ is a Brill--Noether stable object and $T_0$ is a torsion sheaf of dimension $0$.
\end{proposition}

With this proposition, we can prove the inequality for negative slope:

\begin{lemma}
For a Brill--Noether stable object $E\in \Coh^{0,H}(X)$ with $\nu_{BN}(E)\in[-\frac{1}{2},0)$, we have $Q_{0,0}(E)\geq0$.
\end{lemma}

\begin{proof}

Let $E\in \mathrm{Coh}^{0,H}(X)$ be a Brill--Noether stable object with $\nu_{BN}(E)\in[-\frac{1}{2},0)$. Then $\mathbb{D}(E)$ fits into the distinguished triangle 
\begin{equation}
\overline{E}\rightarrow\mathbb{D}(E)\rightarrow T_0[-1]\rightarrow \overline{E}[1]   
\end{equation}
for some Brill--Noether stable object $\overline{E}$ and some torsion sheaf  $T_0$ of dimension $0$ by Proposition \ref{prop:dual}. The Chern characters are related by 
$$\ch_1(\overline{E})=\ch_1(\mathbb{D}(E))=\ch_1(E)
\;\;\text{and}\;\;\ch_2(\overline{E})=-\ch_2(E).$$
Thus we have $\nu_{BN}(\overline{E})\in(0,\frac{1}{2}]$. Thus we have
$$Q_{0,0}(E)=Q_{0,0}(\mathbb{D}(E))=Q_{0,0}(\overline{E})+6H^2\ch_1(E)\ch_3(T_0)\geq0.$$
\end{proof}

\begin{proof}[Proof of Proposition \ref{prop:main3}]

The remaining case is when $\nu_{BN}(E)=H\ch_2(E)=0$. To show this, we define the object $\Tilde{E}$ to be
$$\Tilde{E}=\mathrm{Cone}(O_X\otimes \Hom(O_X,E)\xrightarrow{ev} E).$$
If $\tilde{E}$ is $\nu_{\alpha,0,H}$-tilt semistable for some $\alpha>0$, then by Theorem \ref{thm:main5} we have 
$$\frac{H^2\ch_1(\Tilde{E})}{H^3\rk(\Tilde{E})}\notin(-\frac{1}{\sqrt{5}},0].$$
This implies 
$$H^3\rk(\Tilde{E})\geq-\sqrt{5}H^2\ch_1(\Tilde{E}).$$
Otherwise, for any $\delta>0$, $\Tilde{E}$ is destablised by some $F_{\delta}$ when considering the $\nu_{\delta^2,\delta,H}$-stability condition. We assume $\delta$ is small enough such that $E$ is $\nu_{\delta^2,\delta,H}$-tilt stable because $E$ is Brill--Noether stable. As $\Hom(F_\delta,\Tilde{E})\neq0$, we have $\Hom(F_{\delta},E)$ or $\Hom(F_{\delta},O[1])$ is nonzero because $\Tilde{E}$ is the cone of them. Then we have $\nu_{\delta^2,\delta,H}(F_{\delta})\leq\nu_{\delta^2,\delta,H}(O[1])$ or $\nu_{\delta^2,\delta,H}(F_{\delta})\leq\nu_{\delta^2,\delta,H}(E)$. By Theorem \ref{thm:main5}, when $\delta<\frac{3}{7}$,
$$\nu_{\delta^2,\delta,H}(E)\leq\nu_{\delta^2,\delta,H}(O_X[1])=\delta.$$
This in total gives $\nu_{\delta^2,\delta,H}(F_{\delta})\leq\nu_{\delta^2,\delta,H}(O[1])=\delta.$ Note that the equality only holds when $F_\delta=O_X[1]$ because both are stable, so we have $\nu_{\delta^2,\delta,H}<\delta$.
Now assume $F_\delta$ has the largest slope among all destabilising objects. Then, by the above argument, the HN filtration for $\Tilde{E}$ has factors $E_i$, with each slope smaller than $\delta$. By the wall property, there is an $\alpha_i$ such that $E_i$ is $\nu_{\alpha_i,0,H}(E_i)$-tilt stable and $\nu_{BN}(E_i)<\delta$. Then by Theorem \ref{thm:main5},
$$\frac{H^2\ch_1(E_i)}{H^3\rk(E_i)}\notin\left[\frac{4\delta-\sqrt{16\delta^2+5}}{5},0\right].$$
As $\delta\rightarrow0$, we have 
$$\frac{H^3\rk(E_i)}{H^2\ch_1(E_i)}\geq-\sqrt{5}.$$
Thus, in any case, we have
$$\hom(O_X,E)\leq \rk(E)+\frac{\sqrt{5}}{8}H^2\ch_1(E).$$
Taking $\overline{E}\rightarrow\mathbb{D}(E)\rightarrow T_0[-1]$, we have
\begin{align*}
\hom(O_X,E[2])&=\hom(E,O_X[1])=\hom(\mathbb{D}(O_X[1]),\mathbb{D}(E))=\hom(O_X,\mathbb{D}(E))\\
&=\hom(O_X,\overline{E})\leq\frac{\sqrt{5}}{8}H^2\ch_1(E)-\rk(E),
\end{align*}
where the first equality is due to Serre duality and being Calabi--Yau, the second is due to duality, the third is straightforward, the fourth is due to the fact that $T_0$ is torsion of dimension $0$ and the inequality is due to $\overline{E}$ being Brill--Noether stable.

Thus, by HRR, we have
$$\ch_3(E)+\frac{7}{12}H^2\ch_1(E)=\chi(E)\leq \hom(O_X,E)+\hom(O_X,E[2])\leq\frac{\sqrt{5}}{4}H^2\ch_1(E).$$
This implies $\ch_3(E)<0$ and we are done.
\end{proof}

\section{Clifford Type Inequality for Curves $X_{2,2,2,4}$}

\subsection{Bound for the Wall}

Let $C:=X_{2,2,2,4}\subset X_{2,2,4}\subset X_{2,4}\subset\mathbb{P}^5$ be the curve of generic smooth complete intersection of $X$ with two quadratic hypersurfaces. By the adjunction formula, we know that the canonical bundle of $C$ is $O_C((-6+2+2+2+4)H)=O_C(4H)$. Then the degree of the canonical bundle is $\deg(O_C(4H))=4*2*2*2*4=128$ and thus the genus is $g=65$ by the formula $2g-2=\deg(O_C(4H))$. Let $S:=X_{2,2,2}$ be a general K3 surface given as the complete intersection of three quadratic hypersurfaces containing this curve of Picard number 1, with three quadratic hypersurfaces coming from the complete intersection that gave us $C$, and let $\imath:C\rightarrow S$ be the embedding. Such a pair $(C,S)$ always exists by \cite[Theorem 1]{ravindra2009noether}. Let $E$ be a slope semistable vector bundle on $C$ with rank $r$ and degree $d$. Then for sufficiently large $\alpha$, the object $\imath_*E$ is a $\sigma_{\alpha,0,H}$-semistable object (\cite[Theorem 3.11]{maciocia2014computing}). We would like to detect its first wall. Suppose the first wall intersects $\Gamma$ at $(\beta_1,\Gamma(\beta_1))$ and $(\beta_2,\Gamma(\beta_2))$ (or the vertical line segments in Proposition \ref{prop:k3w} $(\beta_1,\alpha_1),(\beta_2,\alpha_2)$), where $\beta_1<\beta_2$. By Grothendieck--Riemann--Roch, we have 
$$\ch(\imath_*E)=(0,4rH,d+(1-g)r).$$
We use $\mu$ to denote the slope of $E$ on $C$. In this case, the curve $\Gamma$ (Figure \ref{graph:gamma}) is 
\begin{equation*}
    \Gamma(x)=
    \begin{cases}
      4x^2-1+(x-n)^2 & \text{if }x\in[n-\frac{1}{2},n)\cup(n,n+\frac{1}{2}], n\in\mathbb{Z}\\
      4x^2 & \text{if }x\in\mathbb{Z}.
    \end{cases}
\end{equation*}

\begin{lemma}\label{lem:walb}
When $\mu\in[0,64]$, we have 
\begin{enumerate}
    \item When $\mu\in[0,\frac{256}{3}-\frac{32\sqrt{61}}{3})$, the object $\imath_*E$ is Brill--Noether semistable.
    \item When $\mu\in[31,32]$, then we have $$\beta_1\geq\frac{\mu}{32}-4\;\;\text{and}\;\;\beta_2\leq1.$$
    \item When $\mu\in[32,33]$, then we have $$\beta_1\geq -3\;\;\text{and}\;\;\beta_2\leq\frac{\mu}{32}.$$
    \item When $\mu\in[63,64]$, then we have$$\beta_1\geq \frac{\mu}{32}-4\;\;\text{and}\;\;\beta_2\leq2.$$
    \item Otherwise, we have $$\beta_1\geq\frac{\mu}{32}-4,\;\;\text{and}\;\; \beta_2\leq\frac{\mu}{32}.$$
\end{enumerate}
\end{lemma}

\begin{proof}
The proof is similar to the proof of  \cite[Lemma 3.1]{feyzbakhsh2018higher} and \cite[Lemma 4.10]{li2019stability}. On the wall, there is a destablising sequence $0\rightarrow F_2\rightarrow\imath_*E\rightarrow F_1\rightarrow0.$ This sequence actually happens in the heart $\Coh^0(S)$. Thus we have the following long exact sequence for cohomology sheaves:
$$0\rightarrow H^{-1}(F_1)\rightarrow F_2\rightarrow\imath_*E\rightarrow H^0(F_1)\rightarrow 0$$
$$\rk\;\;\;\;\;\;\;\;\;\;s\;\;\;\;\;\;\;\;\;\;\;\;\;\;s\;\;\;\;\;\;\;\;\;0\;\;\;\;\;\;\;\;\;\;\;\;\;0\;\;\;\;\;\;\;\;\;\;\;\;\;\;$$
$$\ch_1\;\;\;\;\;\;\;\;d_1H\;\;\;\;\;\;\;\;\;d_2H\;\;\;\;\;4rH\;\;\;\;\;\;\;\;4aH\;\;\;\;\;\;\;\;\;$$
The left side is $0$ since  $H^{-1}(\imath_*E)=0$. Now we have two cases: $s=0$ or $s\neq0$.

Suppose $s=0$. Then $H^{-1}(F_1)=0$, because this term is torsion-free. Since $\imath_*E$ is supported on $C$, the other two supports are contained in $C$. Since $F_2$ destablises $\imath_*E$, $F_2$ and $\imath_*E$ have the same tilt slope. Thus we get $\ch(F_2)=k\ch(\imath_*E)$, which contradicts the destabilising sequence being on the first wall. So this case cannot happen.

Suppose $s\neq0$. Let $T(F_2)$ be the maximal torsion subsheaf of $F_2$. Suppose $\ch_1(T(F_2))=4tH$. Since $E$ is of rank $r$, we have
$$r-a\leq \rk(\imath^*T(F_2))+\rk(\imath^*(F_2/T(F_2)))=s+t.$$
From this, we get the following inequality:
$$\frac{H\ch_1(F_2/T(F_2))}{sH^2}-\frac{H\ch_1(H^{-1}(F_1))}{sH^2}=\frac{d_2-4t-d_1}{s}=\frac{4r-4a-4t}{s}\leq4.$$

By Proposition \ref{prop:k3w}, we have $F_1$ is semistable with the same slope as $\imath_*E$ on the wall. In particular, if $-4<\beta_1$, it is in the heart $\mathrm{Coh}^{\beta_1+\epsilon}$ when $\epsilon\rightarrow0^+$. Thus
$$\frac{H\ch_1(H^{-1}(F_1))}{H^2s}=\frac{d_1}{s}\leq\beta_1.$$
Thus
$$\frac{H\ch_1(F_2/T(F_2))}{H^2s}\leq4+\beta_1.$$
Suppose $\beta_1\leq-4$, then $\frac{d_1}{s}\leq-4$, and thus
$$\frac{H\ch_1(F_2/T(F_2))}{H^2s}\leq0.$$ This contradicts the assumption that $F_2/T(F_2)\in\Coh^0(S)$.

On the other hand, by Proposition \ref{prop:k3w}, we know that
$$\frac{H\ch_1(F_2/T(F_2))}{H^2s}\geq\beta_2.$$
Thus we have that
$$\beta_2-\beta_1\leq 4\text{ and }-4\leq\beta_1,\beta_2\leq4.$$
Also by Proposition \ref{prop:k3w}, we know that the wall needs to have slope of $\frac{\mu}{4}-16$ by the Chern character of $\imath_*E$. However, there are two separate cases we need to consider here, where the first case is that the wall ends on $\Gamma$, and the second case is the wall ends at the vertical line segment from $(n,\frac{H^2}{2}n^2)$ to $\left(n,\frac{H^2}{2}n^2-1\right)$. Since we want to detect the largest range of the wall, the maximum can happen when $\beta_2-\beta_1=4$. Therefore from now on we assume that $\beta_2-\beta_1$ is equal to $4$.

\textbf{Case 1:} The object $\imath_*E$ is Brill--Noether semistable. In this case the line $l$ intersects the y-axis below zero, where the line $l$ is the line with slope 
$$\frac{\Gamma(\beta_2)-\Gamma(\beta_1)}{\beta_2-\beta_1}=\frac{H^2\ch_2(\imath_*E)}{H\ch_1(\imath_*E)}=\frac{\mu}{4}-16$$
and intersecting the curve at $\beta_1$ and $\beta_2=\beta_1+4$. Then, similar to the case 2 below, we have that the line intersects $\Gamma$ at $\beta_2=\frac{\mu}{32}$ and $\beta_1=\frac{\mu}{32}-4$. As the slope is small in this case, we can assume $\Gamma(x)=5x^2-1$ near $\beta_2$. Then the intersection point with the y-axis is
$$t:=-\frac{3}{1024}\mu^2+\frac{\mu}{2}-1,$$
and the requirement for $\mu$ is $t<0$, which is equivalent to $\mu\in[0,\frac{256}{3}-\frac{32\sqrt{61}}{3})$.

\textbf{Case 2:} The end point is on $\Gamma$. Then we have
$$\frac{\Gamma(\beta_2)-\Gamma(\beta_1)}{\beta_2-\beta_1}=\frac{H^2\ch_2(\imath_*E)}{H\ch_1(\imath_*E)}=\frac{\mu}{4}-16.$$
Substituting $\beta_2=\beta_1+4$ and the equation of $\Gamma$, we have
$$\Gamma(\beta_2)-\Gamma(\beta_1)=\Gamma(\beta_2)-\Gamma(\beta_2-4)=\frac{H^2}{2}((\beta_2)^2-(\beta_2-4)^2)=\frac{H^2}{2}(8\beta_2-16).$$
This quantity is equal to $$(\beta_2-\beta_1)(\frac{\mu}{4}-16)=\mu-64.$$
Thus $\beta_2=\frac{\mu}{32}$, and $\beta_1=\frac{\mu}{32}-4$.

\textbf{Case 3}: The end point is on the vertical line segment from $(n,\frac{H^2}{2}n^2)$ to $\left(n,\frac{H^2}{2}n^2-1\right)$. First one notices that the slope is always negative in the range that we are considering. So the special endpoints happen when both endpoints touch the vertical wall. So we just pick $\beta_1=-n$, and the endpoint has a vertical value of $4n^2$. The corresponding $\beta_2$ is $-n+4$, and the vertical minimum value is $4(4-n)^2-1$. By a direct calculation, we get
\begin{align*}
    \beta_2=1\text{ when }\mu\in[31,32], \;\;\;\; \beta_2=2\text{ when }\mu\in[63,64], \;\;\text{and}\;\; \beta_1=-3\text{ when }\mu\in[32,33].
\end{align*}
\end{proof}

\subsection{Upper Bound of Global Sections}

An upper bound for the global sections of semistable objects on K3 surfaces is already known in \cite{feyzbakhsh2019mukai}. However, we need to produce a better bound to get our final result. We will see that in some regions, the two results coincide. The idea here is that we consider proper spherical twist of the object. The spherical twist of the object is semistable with respect to some generic stability condition $\sigma_{\alpha,\beta}$ (refer to the proof of Proposition \ref{prop:boundsbnobj}) and thus the Chern characters need to lie under the curve $\Gamma$. Let $S$ be the K3 surface we mentioned at the beginning of this section, which is the complete intersection of three quadratic hypersurfaces. The following proposition is essential to the calculation.

\begin{proposition}[{\cite[Proposition 3.4]{feyzbakhsh2019mukai}}]
Let $X$ be a K3 surface. Let $E\in \mathrm{Coh}^{0,H}(X)$, then there exists $\epsilon>0$ such that the HN filtration for $E$ is the same sequence with respect to all $\sigma_{\alpha,0,H}$ for $0<\alpha<\epsilon$, and is denoted by 
$$0=\Tilde{E}_0\subset\Tilde{E}_1\subset...\subset\Tilde{E}_n=E.$$

\end{proposition}

\begin{proposition}
\label{prop:boundsbnobj}
Let $F$ be a $\nu_{BN}$-semistable object in $\Coh^{0,H}(S)$. Then we have
\begin{equation}
 \hom(\cO_S,F)\leq
    \begin{cases}
      \rk(F)+\frac{2}{3}\frac{H\ch_1(F)}{H^2}+\frac{7}{6}\ch_2(F)\\
      -\frac{1}{6}\sqrt{\left(4\frac{H\ch_1(F)}{H^2}+\ch_2(F)\right)^2-60\left(\frac{H\ch_1(F)}{H^2}\right)^2} &\text{if } \frac{H^2\ch_2(F)}{H\ch_1(F)}\in[\frac{11}{2},\frac{15}{2})\cup(8,\frac{97}{10}]\\
      \rk(F)+\ch_2(F)+\frac{5(H\ch_1(F))^2}{8(H^2\ch_2(F)+2H\ch_1(F))} &\text{if }\frac{H^2\ch_2(F)}{H\ch_1(F)}\in[\frac{1}{2},3)\cup(4,\frac{11}{2}]\\
      \rk(F)+\frac{\ch_2(F)}{2}+\frac{1}{2}\sqrt{\ch_2(F)^2+20\left(\frac{H\ch_1(F)}{H^2}\right)^2} & \text{if }\frac{H^2\ch_2(F)}{H\ch_1(F)}\in[-\frac{1}{2},\frac{1}{2}]\\
      \rk(F)+\frac{5(H\ch_1(F))^2}{H^2(2H\ch_1(F)-H^2\ch_2(F))} & \text{if } \frac{H^2\ch_2(F)}{H\ch_1(F)}\in[-\frac{11}{2},-4)\cup(-3,-\frac{1}{2}]\\ 
      \rk(F)+\frac{2}{3}\frac{H\ch_1(F)}{H^2}-\frac{1}{6}\ch_2(F)\\-\frac{1}{6}\sqrt{\left(\ch_2(F)-4\frac{H\ch_1(F)}{H^2}\right)^2-60\left(\frac{H\ch_1(F)}{H^2}\right)^2} & \text{if }\frac{H^2\ch_2(F)}{H\ch_1(F)}\in[-\frac{97}{10},-8)\cup(-\frac{15}{2},-\frac{11}{2}]\\
      \rk(F)+\frac{3}{8}\frac{H\ch_1(F)}{H^2}-\frac{\ch_2(F)}{16}\\
      -\frac{1}{16}\sqrt{\left(\ch_2(F)-6\frac{H\ch_1(F)}{H^2}\right)^2-160\left(\frac{H\ch_1(F)}{H^2}\right)^2} & \text{if } \frac{H^2\ch_2(F)}{H\ch_1(F)}\in[-\frac{193}{14},-12)\cup(-\frac{35}{3},-\frac{97}{10}]\\
      \rk(F)+\frac{4}{15}\frac{H\ch_1(F)}{H^2}-\frac{\ch_2(F)}{30}\\
      -\frac{1}{30}\sqrt{\left(\ch_2(F)-8\frac{H\ch_1(F)}{H^2}\right)^2-300\left(\frac{H\ch_1(F)}{H^2}\right)^2} & \text{if } \frac{H^2\ch_2(F)}{H\ch_1(F)}\in [-\frac{107}{6},-16)\cup(-\frac{63}{4},-\frac{193}{14}]\\
      \rk(F)-\frac{(H\ch_1(F))^2}{16\ch_2(F)} & \text{if } \frac{H^2\ch_2(F)}{H\ch_1(F)}\in[-4n,\frac{1-4n^2}{n}]\text{ and }n\in\mathbb{Z}_{>0} \\
      \rk(F)+\ch_2(F)+\frac{(H\ch_1(F))^2}{16\ch_2(F)} & \text{if } \frac{H^2\ch_2(F)}{H\ch_1(F)}\in[\frac{4n^2-1}{n},4n]\text{ and }n\in\mathbb{Z}_{>0} 
    \end{cases}       
\end{equation}
\end{proposition}

\begin{figure}
\centering
\begin{tikzpicture}
[
    dot/.style={
        draw=black,
        fill=blue!90,
        circle,
        minimum size=3pt,
        inner sep=0pt,
        solid,
    },
]
\begin{axis}[
    axis lines = center,
    xlabel = \(\frac{H\ch_1}{H^2\rk}\),
    ylabel = {\(\frac{\ch_2}{\rk}\)},
]

\addplot [
    domain=-0.5:0.5, 
    samples=100, 
    color=blue,
    name path=gamma
    ]
    {5*x^2  - 1};
\addplot [
    domain=0.5:1.5, 
    samples=100, 
    color=blue,
    name path=gammap1
    ]
    {5*x^2  - 2*x};
\addplot [
    domain=-1.5:-0.5, 
    samples=100, 
    color=blue,
    ]
    {5*x^2  +2*x};
\addplot [
    domain=1.5:2.5, 
    samples=100, 
    color=blue,
    name path=gammap2
    ]
    {5*x^2  - 4*x+3};
\addplot [
    domain=-2.5:-1.5, 
    samples=100, 
    color=blue,
    ]
    {5*x^2  + 4*x+3};
\addplot [
    domain=2.5:3.5, 
    samples=100, 
    color=blue,
    name path=gammap3
    ]
    {5*x^2  - 6*x+8};
\addplot [
    domain=-3.5:-2.5, 
    samples=100, 
    color=blue,
    name path=gamman3
    ]
    {5*x^2  + 6*x+8};
\addplot [
    domain=3.5:4, 
    samples=100, 
    color=blue,
    ]
    {5*x^2  - 8*x+15};
\addplot [
    domain=-4:-3.5, 
    samples=100, 
    color=blue,
    ]
    {5*x^2  + 8*x+15};
\addplot [
    domain=-2:4, 
    samples=100, 
    color=red,
    name path=pol
    ]
    {10*x};   
\addplot [
    domain=-4:2, 
    samples=100, 
    color=green,
    name path=nel
    ]
    {-10*x};   
 \fill [name intersections={of=pol and gammap3}] (intersection-1) circle (1.5pt)
            coordinate (a) node [dot,label=below:{$(x_0,\Gamma(x_0))$}] {};

\fill [name intersections={of=nel and gamman3}] (intersection-1) circle (1.5pt)
            coordinate (b) node [dot,label=below:{$(x_1,\Gamma(x_1))$}] {};
\end{axis}
\end{tikzpicture}
\caption{The $\Gamma$ curve (blue) intersects with positive slope line (red) and negative slope line (green)} \label{graph:gamma}
\end{figure}
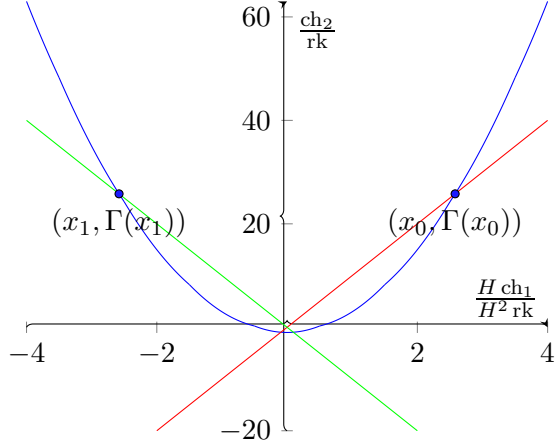

\begin{proof}
The proof is divided into three parts, depending on the sign of $\frac{H^2\ch_2(F)}{H\ch_1(F)}$ and whether it is close to $4n$ for some integer $n$.

\textbf{First Case} When $\frac{H^2\ch_2(F)}{H\ch_1(F)}>0$ and it is not inside $[\frac{4n^2-1}{n},4n]$, we have $\nu_{BN}(O_S[1])=+\infty$. Thus
$$\hom(O_S,F[-1-i])=\hom(O_S[1+i],F)=0$$
for $i\geq0$. On the other hand, there exists some $(\alpha,\beta)$ on the line through $0$ and $pr(F)$ such that $\alpha>\Gamma(\beta)$. Thus
$$\hom(O_S,F[2+i])=\hom(F,O_S[-i])=0$$
for $i\geq0$, because $O_S[1]$ and $F$ both are $\sigma_{\alpha,\beta}$-semistable of same slope by the nesting wall theorem. Thus we have
$$\chi(O_S,F)=\hom(O_S,F)-\hom(O_S,F[1]).$$

We consider the object
$$\Tilde{F}[1]=\mathrm{Cone}(F\xrightarrow{can} O[1]\otimes \Hom(F,O[1])^*).$$ 
Then $\Tilde{F}$ is $\sigma_{\alpha,\beta}$-semistable for $\beta<0$. We claim that $pr(\tilde{F})$ is below the curve $\Gamma$. Otherwise, suppose $pr(\tilde{F})$ is not below the curve $\Gamma$. Consider the HN factors $F_i$ of $\tilde{F}$ with respect to $\sigma_{\alpha,\beta}$. Then by Lemma \ref{lem:wall}, the $pr(F_i)$ are on the line passing through $(\alpha,\beta)$ and $\tilde{F}$. As the segment of this line above $\Gamma$ is convex, there is at least one HN factor $F_j$ is above $\Gamma$. However, by Proposition \ref{prop:gamma}, $F_j$ cannot lie above the curve $\Gamma$. So the point $pr(\Tilde{F})$ is below the curve $\Gamma$. We consider the line passing through $0$ with slope $k$ (Figure \ref{graph:gamma}), where $k=\frac{H^2\ch_2(F)}{H\ch_1(F)}$.  Then it is obvious that $pr(O[1])$, $pr(F)$, and $pr(\Tilde{F}[1])$ are on this line. Let $x_0$ be the intersection point of the line with $\Gamma$. Depending on the slope, we can solve $x_0$ explicitly:

When $\frac{H^2\ch_2(F)}{H\ch_1(F)}\in[\frac{11}{2},\frac{15}{2})\cup(8,\frac{97}{10}]$, we have $\Gamma(x)=5x^2-4x+3$. By solving the intersection equation with the line, we have 
$$x_0=\frac{4+k+\sqrt{((4+k)^2-60)}}{10}.$$
Thus, since $p_H(\Tilde{F}[1])$ is below $\Gamma$, we have
$$\frac{H\ch_1(\Tilde{F}[1])}{H^2\rk(\Tilde{F}[1])}\geq x_0.$$
(The case $\frac{H\ch_1(\Tilde{F}[1])}{H^2\rk(\Tilde{F}[1])}<0$ is also contained in this case.)
Here, $\ch_1(\tilde{F}[1])=-\ch_1(\tilde{F})$, $\rk(\tilde{F}[1])=-\rk(\tilde{F})$ and $\rk(\tilde{F})=\hom(O_S,F[1])+\rk(F)>0$.
This implies that 
$$\frac{H\ch_1(F)}{H^2x_0}\geq \rk(\Tilde{F})=\hom(O_S,F[1])+\rk(F),$$
and considering this with 
$$\hom(O_S,F)=\chi(O_S,F)+\hom(O_S,F[1])$$
we get the conclusion.

When $\frac{H^2\ch_2(F)}{H\ch_1(F)}\in[\frac{1}{2},3)\cup(4,\frac{11}{2}]$, we use the same calculation method with $\Gamma(x)=5x^2-2x$ and we get the conclusion.

\textbf{Second case} When $\frac{H^2\ch_2(F)}{H\ch_1(F)}<0$ and it is not inside $[-4n,\frac{1-4n^2}{n}]$, we know that there exist some $(\alpha,\beta)$ on the line through $0$ and $pr(F)$ with $\alpha>\Gamma(\beta)$ and by nesting wall theorem, $O[1]$ and $F$ are $\sigma_{\alpha,\beta}$-semistable. If we consider 
$$\Tilde{F}=\mathrm{Cone}(\Hom(O_X,F)\otimes O\xrightarrow{ev} F),$$
we get that $\Tilde{F}$ is $\sigma_{\alpha,\beta}$-semistable for $\beta>0$. By a similar argument to the first case, we know that the reduced character $pr(\tilde{F})$ is below the curve $\Gamma$. Consider the line passing through $0$ and $pr(F)$ (Figure \ref{graph:gamma}). Then $pr(\Tilde{F})$ is on the same line. Let $x_1$ be the intersection of the line and $\Gamma$ on the left-hand side. We can solve $x_1$ explicitly depending on the slope of the line:

When $\frac{H^2\ch_2(F)}{H\ch_1(F)}\in [-\frac{107}{6},-16)\cup(-\frac{63}{4},-\frac{193}{14}]$, let the line be $y=kx$, where $k=\frac{H^2\ch_2(F)}{H\ch_1(F)}$. Recall that when $x\in[n-\frac{1}{2},n+\frac{1}{2}]$, the curve $\Gamma$ has the form
\begin{equation*}
    \Gamma(x)=
    \begin{cases}
      4x^2-1+(x-n)^2 & \text{if }x\neq n\\
      4x^2 & \text{if }x=n
    \end{cases}
\end{equation*}
Then by solving the equation, we get
$$x_1=\frac{k-8-\sqrt{(8-k)^2-300}}{10}.$$
Then we have the requirement that 
$$\frac{H\ch_1(\Tilde{F})}{H^2\rk(\Tilde{F})}\leq x_1.$$
This in turn tells us that
$$\frac{H\ch_1(\Tilde{F})}{H^2x_1}\leq \rk(\Tilde{F})=\rk(F)-\hom(O_S,F),$$
and then this gives
$$\hom(O_S,F)\leq \rk(F)+\frac{4}{15}\frac{H\ch_1(F)}{H^2}-\frac{\ch_2(F)}{30}-\frac{1}{30}\sqrt{\left(\ch_2(F)-8\frac{H\ch_1(F)}{H^2}\right)^2-300\left(\frac{H\ch_1(F)}{H^2}\right)^2}.$$

For $\frac{H^2\ch_2(F)}{H\ch_1(F)}$ inside the other range, a similar calculation is done with the appropriate expression for $\Gamma$ used.

\textbf{Third case} When the slope $\frac{H^2\ch_2(F)}{H\ch_1(F)}\in[-4n,\frac{1-4n^2}{n}]$ or $[\frac{4n^2-1}{n},4n]$. In this case, we just use a proper spherical twist as above and use the Bogomolov inequality. If moreover $\frac{H^2\ch_2(F)}{H\ch_1(F)}>0$, then we consider $\tilde{F}$ as in the first case. Because $\tilde{F}$ is $\sigma_{\alpha,\beta}$-semistable for some $(\alpha,\beta)$, we have
$$(H\ch_1(F))^2-2H^2(\ch_2(F))(\rk(F)+\mathrm{ext}^1(O_S,F))\geq 0,$$
hence
$$ \mathrm{ext}^1(O_S,F)\leq \rk(F)+\frac{(H\ch_1(F))^2}{2H^2\ch_2(F)}+\ch_2(F).$$
If $\frac{H^2\ch_2(F)}{H\ch_1(F)}<0$, we consider $\tilde F$ as in the second case. Because $\tilde{F}$ is $\sigma_{\alpha,\beta}$-semistable for some $(\alpha,\beta)$, we have
$$(H\ch_1(F))^2-2H^2(\ch_2(F))(\rk(F)-\hom(O_S,F))\geq 0,$$
$$\implies \hom(O_S,F)\leq \rk(F)-\frac{(H\ch_1(F))^2}{2H^2\ch_2(F)}.$$

\end{proof}

\begin{remark}
We see that the bound for the range $[-\frac{1}{2},\frac{1}{2}]$ works for all other ranges. This is the same bound given by \cite{feyzbakhsh2019mukai}.
\end{remark}

Next, we would like to show some kind of convexity of the above bound. Let

\begin{equation}
  \spadesuit(x,y)=
    \begin{cases}
      \frac{2}{3}\frac{y}{H^2}+\frac{7}{6}x
      -\frac{1}{6}\sqrt{\left(4\frac{y}{H^2}+x\right)^2-60\left(\frac{y}{H^2}\right)^2} &\text{if } \frac{H^2x}{y}\in[\frac{11}{2},\frac{15}{2})\cup(8,\frac{97}{10}]\\
      x+\frac{5y^2}{8(H^2x+2y)} &\text{if }\frac{H^2x}{y}\in[\frac{1}{2},3)\cup(4,\frac{11}{2}]\\
      \frac{x}{2}+\frac{1}{2}\sqrt{x^2+20\left(\frac{y}{H^2}\right)^2} & \text{if }\frac{H^2x}{y}\in[-\frac{1}{2},\frac{1}{2}]\\
      \frac{5y^2}{H^2(2y-H^2x)} & \text{if } \frac{H^2x}{y}\in[-\frac{11}{2},-4)\cup(-3,-\frac{1}{2}]\\ 
      \frac{2}{3}\frac{y}{H^2}-\frac{1}{6}x-\frac{1}{6}\sqrt{\left(x-4\frac{y}{H^2})^2-60(\frac{y}{H^2}\right)^2} & \text{if }\frac{H^2x}{y}\in[-\frac{97}{10},-8)\cup(-\frac{15}{2},-\frac{11}{2}]\\
      \frac{3}{8}\frac{y}{H^2}-\frac{x}{16}-\frac{1}{16}\sqrt{\left(x-6\frac{y}{H^2}\right)^2-160\left(\frac{y}{H^2}\right)^2} & \text{if } \frac{H^2x}{y}\in[-\frac{193}{14},-12)\cup(-\frac{35}{3},-\frac{97}{10}]\\
      \frac{4}{15}\frac{y}{H^2}-\frac{x}{30}-\frac{1}{30}\sqrt{\left(x-8\frac{y}{H^2}\right)^2-300\left(\frac{y}{H^2}\right)^2} & \text{if } \frac{H^2x}{y}\in [-\frac{107}{6},-16)\cup(-\frac{63}{4},-\frac{193}{14}]\\
      -\frac{y^2}{16x} & \text{if } \frac{H^2x}{y}\in[-4n,\frac{1-4n^2}{n}]\text{ and }n\in\mathbb{Z}_{>0} \\
      x+\frac{y^2}{16x} & \text{if } \frac{H^2x}{y}\in[\frac{4n^2-1}{n},4n]\text{ and }n\in\mathbb{Z}_{>0}
    \end{cases}       
\end{equation}

\begin{lemma}\label{lem:con}
Let $O=(0,0)$ be the origin, and let $P=(x_p,y_p)$ and $Q=(x_q,y_q)$ be two points on the upper half plane such that $\frac{x_p}{y_p}>\frac{x_q}{y_q}$ and $y_p<y_q$. Consider a sequence of points $P_0=O,P_1,...,P_n=Q$ on the upper half plane and inside the triangle $OPQ$. In addition, assume that the points $P_0,...,P_n$ form a convex polygon. If we consider the sum
$$\sum_{i=0}^{n-1}\spadesuit(\overrightarrow{P_{i}P_{i+1}}),$$
it can achieve its maximum only when $n=1$ or $n=2$. When $n=2$, we can choose $P_1$ on $OP$ or $PQ$, unless $\frac{x_1}{y_1}=m$, or $\frac{4m^2-1}{m}$; or $\frac{x_q-x_1}{y_q-y_1}=m$ or $\frac{4m^2-1}{m}$ for some nonzero integer $m$.
\end{lemma}

\begin{proof}
Here we use the trick in \cite[Section 2.2]{feyzbakhsh2018higher}. The first part is basically the same as in \cite[Lemma 4.11]{li2019stability}, because the essence of the proof is that in all the cases, the function is homogeneous of degree 1. This implies that we can reduce to the case $n\leq 2$. Now we consider a triangle $OAB$, with slope $\frac{x(A)}{y(A)}>\frac{x(B)}{y(B)}$, such that slopes of $\overrightarrow{OA}, \overrightarrow{OB}$, or slopes of $\overrightarrow{AB}, \overrightarrow{OB}$ fall in the same region. By calculating the derivative, we get a weak triangular inequality, that is to say $$\spadesuit(\overrightarrow{OA})+\spadesuit(\overrightarrow{AB})\geq\spadesuit(\overrightarrow{OB}).$$
Now we consider the triangle $OP'Q$ inside the triangle $OPQ$. Then by extending the line $OP'$ or $QP'$, we get a new small triangle. As the function $\spadesuit$ is linear when the slope is fixed, we can just consider the new small triangle. Then by the weak triangular inequality, we get the conclusion. The only thing that needs to be proven is the case in which the changing point is not equal to $\frac{4m^2-1}{m}$ or $4m$. Let $P'$ be a point that does not coincide with $P$ and the slope of $OP'$ (or $P'Q$) is at the changing point not equal to $\frac{4m^2-1}{m}$ or $4m$. Then the value $\spadesuit(\overrightarrow{OP'})+\spadesuit(\overrightarrow{P'Q})$ achieves the same value when the sum is calculated by the functions in different regions. Also by the weak triangular inequality, we get the conclusion, and thus we finish the proof.
\end{proof}

With the above lemma, we can give a Clifford type inequality for the curve $C=X_{2,2,2,4}$. In the next proof, we make a change of the coordinate of $\spadesuit$ to be $(x',y')=(x,\frac{y}{H^2})$. By abuse of notation, we still use $\spadesuit$ to denote it.

\begin{theorem}\label{thm:main4}
Let $F$ be a semistable vector bundle on $C$ of rank $r$, degree $d$, and slope $\mu=\frac{d}{r}$. Then we have the following inequality (Figure \ref{graph:Clifford}):
\begin{equation}
    h^0(F)\leq
    \begin{cases}
      \frac{64r^2}{64r-d} & \text{if }
      \mu\in[0,\frac{256}{3}-\frac{32\sqrt{61}}{3})\\
      r+\frac{5d^2}{1024r} & \text{if } \mu\in[\frac{256}{3}-\frac{32\sqrt{61}}{3},16]\\
      \frac{5d^2}{1024r}+5r-\frac{d}{8} & \text{if } \mu\in[48,\frac{576-32\sqrt{69}}{5}]\\
      d-46r & \text{if } \mu \in[\frac{576-32\sqrt{69}}{5},64]
    \end{cases}
\end{equation}
\end{theorem}

\begin{figure}
\centering
\begin{tikzpicture}
[
    dot/.style={
        draw=black,
        fill=blue!90,
        circle,
        minimum size=3pt,
        inner sep=0pt,
        solid,
    },
]
\begin{axis}[
    axis lines = center,
    xlabel = \(\mu\),
    ylabel = {\(\frac{h^0(F)}{r}\)},
    ymin = 0.5,
]

\addplot [
    domain=0:256/3-32*sqrt{61}/3, 
    samples=100, 
    color=blue,
    name path=gamma
    ]
    {64/(64-x)};

\addplot [
    domain=256/3-32*sqrt{61}/3:16, 
    samples=100, 
    color=blue,
    name path=gammap1
    ]
    {1+5/1024*x^2};

\end{axis}
\end{tikzpicture}
\caption{The Clifford type inequality for the curve $C$} \label{graph:Clifford}
\end{figure}
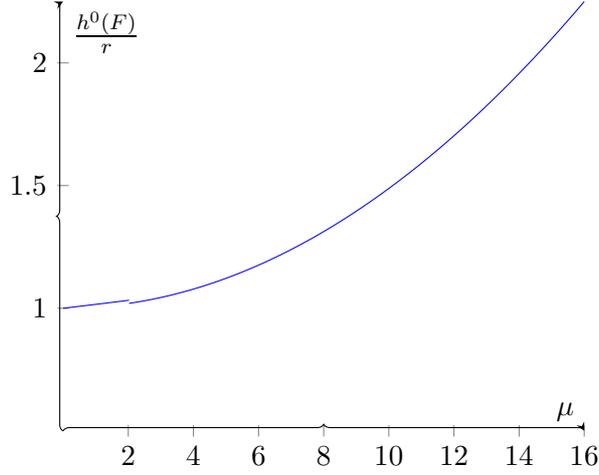

\begin{proof}
If the object $\imath_*F$ is Brill--Noether semistable, then by the Bogomolov inequality, with the same argument as in Proposition \ref{prop:boundsbnobj} , we have 
$$h^0(F)\leq \rk(\imath_*F)-\frac{(H\ch_1(\imath_*F))^2}{2H^2\ch_2(\imath_*F)}= \frac{64r^2}{64r-d}.$$ 
In particular, by Lemma \ref{lem:walb}, the bound holds for $\mu\in[0,\frac{256}{3}-\frac{32\sqrt{61}}{3})$. We may assume that there is a wall as that in Lemma \ref{lem:walb} for $\imath_*F$ for the rest of the argument.

For $\mu\in[\frac{256}{3}-\frac{32\sqrt{61}}{3},16]$, by Proposition \ref{lem:walb}, we know that the wall is inside the range $\frac{\mu}{32}$ and $\frac{\mu}{32}-4$. Let $m$ be the number of  HN factors of $\imath_*F$ with respect to the Brill--Noether stability condition as that in Proposition  \ref{prop:boundsbnobj}. Then we have:
\begin{align*}
h^0(F)&=\hom(O_S,\imath_*F)\leq\sum_{i=0}^{m-1}\hom(O_S,F_{i+1}/F_{i})\\
&\leq\sum_{i=0}^{m-1}\rk(F_{i+1}/F_i)+\spadesuit(\ch_2(F_{i+1}/F_i),H\ch_1(F_{i+1}/F_i)/H^2)\\
&=\sum_{i=0}^{m-1}\spadesuit(\ch_2(F_{i+1}/F_i),H\ch_1(F_{i+1}/F_i)/H^2).
\end{align*}
On the other hand, by Lemma \ref{lem:con}, we can take the number of HN factors to be less than or equal to $2$. By the explanation at the beginning for the case that $\imath_*F$ is BN-semistable, we may assume that the number of the HN factors is $2$. Let $Q=(d-64r,4r)$. Let $P$ be the point such that the slope of $PQ$ is
$$\frac{\frac{\mu}{32}-4}{\Gamma(\frac{\mu}{32}-4)}=\frac{\frac{\mu}{32}-4}{\frac{H^2}{2}(\frac{\mu}{32}-4)^2-1+(\frac{\mu}{32})^2}$$ 
and the slope of $OP$ is
$$\frac{\frac{\mu}{32}}{\Gamma(\frac{\mu}{32})}=\frac{\frac{\mu}{32}}{\frac{H^2}{2}(\frac{\mu}{32})^2-1+(\frac{\mu}{32})^2}=\frac{\frac{\mu}{32}}{5(\frac{\mu}{32})^2-1}.$$ 
By a direct calculation, the point $P$ has the coordinate $(\frac{5d^2}{1024r}-r,\frac{d}{32}).$ Inside this region, $\left(\frac{H^2\ch_2}{H\ch_1}\right)^-(\imath_*(F))\geq-\frac{63}{4}$. Also, we notice that in each case we have 
$$\hom(O_S,F)\leq \rk(F)+\frac{\ch_2(F)}{2}+\frac{1}{2}\sqrt{\ch_2(F)^2+20(H\ch_1)^2}.$$ 
Then, by Lemma \ref{lem:con}, we get $h^0(F)\leq f_1+f_2$, where 
$$f_1=\frac{5d^2}{2048}-\frac{r}{2}+\frac{1}{2}\left(\frac{5d^2}{1024r}+r\right)=\frac{5d^2}{1024r}$$
and
$$f_2=\frac{4}{15}\left(4r-\frac{d}{32}\right)-\frac{1}{30}\left(d-64r-\frac{5d^2}{1024r}+r\right)-\frac{\sqrt{\Delta}}{30},$$
where
$$\Delta=\left(d-64r-\frac{5d^2}{1024r}+r-32r+\frac{d}{4}\right)^2-300\left(4r-\frac{d}{32}\right)^2.$$
By a direct calculation, we get $\sqrt{\Delta}=65r-\frac{5}{4}d+\frac{5d^2}{1024r}$, $f_2=r$ and thus we get the second inequality.

When $\mu\in[48,64]$, we let $Q=(d-64r,4r)$ and let $P$ be the point satisfying that the slope of $OP$ is
$$\frac{\frac{\mu}{32}}{\Gamma(\frac{\mu}{32})}=\frac{\frac{\mu}{32}}{\frac{H^2}{2}(\frac{\mu}{32})^2-1+(\frac{\mu}{32}-2)^2}=\frac{\frac{\mu}{32}}{4(\frac{\mu}{32})^2-1+(\frac{\mu}{32}-2)^2}$$
and the slope of $PQ$ is
$$\frac{\frac{\mu}{32}-4}{\Gamma(\frac{\mu}{32}-4)}=\frac{\frac{\mu}{32}-4}{\frac{H^2}{2}(\frac{\mu}{32}-4)^2-1+(\frac{\mu}{32}-2)^2}.$$ By  a direct calculation, we get
$$P=\left(\frac{5d^2}{1024r}-\frac{d}{8}+3r,\frac{d}{32}\right).$$
Thus we have

\begin{align*}
\spadesuit(\overrightarrow{OP})&=\frac{d}{48}+\frac{7}{6}\left(\frac{5d^2}{1024r}-\frac{d}{8}+3r\right)-\frac{1}{6}\sqrt{\left(\frac{d}{8}+\frac{5d^2}{1024r}-\frac{d}{8}+3r\right)^2-60\left(\frac{d}{32}\right)^2}\\
&=\frac{d}{48}+\frac{7}{6}\left(\frac{5d^2}{1024r}-\frac{d}{8}+3r\right)+\frac{1}{6}(3r-\frac{5d^2}{1024r})=\frac{5d^2}{1024r}+4r-\frac{d}{8}\;\; \text{and}   
\end{align*}
$$\spadesuit(\overrightarrow{PQ})=\frac{8}{3}r-\frac{d}{48}-\frac{d}{6}+\frac{32}{3}r+\frac{5d^2}{6\times 1024r}-\frac{d}{48}+\frac{r}{2}-\frac{1}{6}\sqrt{\Delta'},$$
where 
$$\Delta'=\left(d-64r-\frac{5d^2}{1024r}+\frac{d}{8}-3r-4\left(4r-\frac{d}{32}\right)\right)^2-60\left(4r-\frac{d}{32}\right)^2.$$
By a direct calculation, we get $\sqrt{\Delta'}=\frac{5d^2}{1024r}-\frac{5d}{4}+77r$ and thus $\spadesuit(\overrightarrow{OP})+\spadesuit(\overrightarrow{
PQ})=\frac{5d^2}{1024r}+5r-\frac{d}{8}$.
On the other hand, we consider now a point $P'$ such that the slope of $OP'$ is 
$$\frac{2}{\Gamma(2)-(\mu-64)}=\frac{2}{\mu-48}$$
and the slope of $P'Q$ is
$$\frac{-2}{\Gamma(2)}=-\frac{1}{8}.$$
By a direct computation, we get $P'=(d-48r,2r)$.
We also consider a point $P''$ on the line $OP$ that the line $PQ$ has slope $-\frac{1}{8}$. We note that both $P'$ and $P''$ are inside $OP'''Q$, where $P'''$ is a point such that the slope of $OP'''$ is $\frac{1}{8}$ and the slope of $P'''Q$ is $-\frac{1}{8}$.
Accordingly, we get $P'''=(\frac{d}{2}-16r,\frac{d}{16}-2r)$ and thus $\spadesuit(\overrightarrow{OP'''})=d-47r$ and $\spadesuit(\overrightarrow{P'''Q})=r$ and so $\spadesuit(\overrightarrow{OP'''})+\spadesuit(\overrightarrow{P'''Q})\leq d-46r$. By Lemma \ref{lem:con}, we get the last two cases by considering $\max\{d-46r, \frac{5d^2}{1024r}+5r-\frac{d}{8}\}.$
\end{proof}

\begin{remark}
For the Brill--Noether semistable region, one can get a better Clifford type inequality by applying Lemma \ref{prop:boundsbnobj}. One can do a more careful argument to make the break point more precise, and thus get a better Clifford type inequality. But the bound above is enough for our purpose.
\end{remark}

\section{Bogomolov--Gieseker Type Inequality on $X_{2,2,4}$ and $X_{2,4}$}

Now we give a Bogomolov--Gieseker type inequality for $\ch_2$ on $S'=X_{2,2,4}$. By Hirzebruch--Riemann--Roch, we get
$$\chi(E)=\ch_2(E)-H\ch_1(E)+20\ch_0(E)$$

for a coherent sheaf $E$ on $S'$. The following lemma is essential to the calculation.
\begin{lemma}[{\cite[Corollary 4.3]{feyzbakhsh2016stability}}, {\cite[Lemma 5.1]{li2019stability}}]\label{lem:res}
Let $(X,H)$ be a polarized variety of dimension $n=2,3$. Let $E$ be a coherent sheaf in $\mathrm{Coh}^{0,H}(X)$. Suppose there exists $\alpha>0$ and $m\in\mathbb{Z}_{>0}$ such that 
\begin{enumerate}
    \item $E(-mH)[1]$ is in $\mathrm{Coh}^{0,H}(X)$;
    \item both $E$ and $E(-mH)[1]$ are $\nu_{\alpha,0,H}$-tilt stable;
    \item $\nu_{\alpha,0,H}(E)=\nu_{\alpha,0,H}(E(-mH)[1]).$
\end{enumerate}
Then for a generic smooth projective irreducible subvariety $Y\in |mH|$, the restriction $E|_Y$ is $\mu_{H_Y}$-semistable. Moreover, $\rk(E)=\rk(E|_Y)$, $H^{n-2}_Y\ch_1(E|_Y)=mH^{n-1}\ch_1(E)$, and when $n=3$, $\ch_2(E|_Y)=mH\ch_2(E).$
\end{lemma}

\begin{proposition}\label{prop:main5}
Suppose $F\in D^b(S')$, with $\frac{H\ch_1(F)}{H^2\rk(F)}\in(0,1)$, is $\nu_{\alpha,0,H}$-semistable or $\nu_{\alpha',1,H}$-semistable for some $\alpha>0$ or $\alpha'>\frac{1}{2}$, then we have the following Bogomolov--Gieseker type inequality (Figure \ref{graph:BG}):

\begin{equation}\label{inq}
    \frac{\ch_2(F)}{H^2\ch_0(F)}\leq
    \begin{cases}
      \left(\frac{H\ch_1(F)}{H^2\ch_0(F)}\right)^2-\frac{H\ch_1(F)}{H^2\ch_0(F)} & \text{if } \frac{H\ch_1(F)}{H^2\ch_0(F)}\in(0,\frac{4}{3}-\frac{\sqrt{13}}{3}]\\
      \frac{5}{8}\left(\frac{H\ch_1(F)}{H^2\ch_0(F)}\right)^2-\frac{1}{8} & \text{if }\frac{H\ch_1(F)}{H^2\ch_0(F)}\in(\frac{4}{3}-\frac{\sqrt{13}}{3},\frac{1}{2}]\\
      \frac{5}{8}\left(\frac{H\ch_1(F)}{H^2\ch_0(F)}\right)^2-\frac{1}{4}\frac{H\ch_1(F)}{H^2\ch_0(F)} & \text{if } \frac{H\ch_1(F)}{H^2\ch_0(F)}\in(\frac{1}{2},\frac{\sqrt{13}}{3}-\frac{1}{3})\\
      \left(\frac{H\ch_1(F)}{H^2\ch_0(F)}\right)^2-\frac{1}{2} & \text{if }\frac{H\ch_1(F)}{H^2\ch_0(F)}\in[\frac{\sqrt{13}}{3}-\frac{1}{3},1).
    \end{cases}
\end{equation}
\end{proposition}
\begin{proof}
The proof here is similar to \cite{li2019stability}. We prove this by contradiction. The idea is that first to reduce to stable objects by considering Jordan--H\"older factors for $\nu_{\alpha,0,H}$ or $\nu_{\alpha',-1,H}$. Next, by using the Feyzbakhsh's restriction theorem (Lemma \ref{lem:res}) and the Clifford type inequality in Section 4 (Theorem \ref{thm:main4}), we get inequalities as in \ref{inq} for the stable objects, and thus get the result.

\textbf{Reduce to stable objects:}   Suppose there is a $\nu_{\alpha,0,H}$ or $\nu_{\alpha',1,H}$-tilt stable object with $\frac{H\ch_1(F)}{H^2\rk(F)}\in(0,1)$ violating the above inequality. We assume that $F$ is an object with the minimal $\overline{\Delta}_H$ of all such objects. Suppose $F$ becomes strictly $\nu_{\alpha,0,H}$-semistable for some $\alpha>0$ (or $\nu_{\alpha',1,H}$ strictly semistable), then we take a Jordan--H\"older filtration. Since the inequality forms a convex curve in $(0,\frac{1}{2}]$ and $[\frac{1}{2},1)$ separately, there is at least one Jordan--H\"older factor violating the inequality. For example, if $\frac{H\ch_1(F)}{H^2\rk(F)}\in\left(0,\frac{1}{2}\right]$ and we are considering $\nu_{\alpha,0,H}$, then the line passing through $(\alpha,0)$ and $p_H(F)$ has the $\left(0,\frac{H\ch_1(F)}{H^2\rk(F)}\right]$ segment completely above the curve of the proposition. Thus $F$ must have a Jordan--H\"older factor $F_i$ that violates the inequality. Similarly we can show it in the remaining cases. Finally, by Lemma \ref{lem:dis}, we have $\overline{\Delta}_H(F_i)<\overline{\Delta}_H(F)$, which contradicts the minimum assumption.

Now suppose that $F$ becomes strictly $\nu_{\alpha_0,\beta,H}$-semistable on the vertical wall for $\beta=\frac{H\ch_1(F)}{H^2\rk(F)}$. We may assume that $F\in \mathrm{Coh}^{\beta,H}(S')$. If all the Jordan--H\"older factors of $F$ are torsion-free, then there is a $\nu_{\alpha,\beta}$-stable Jordan--H\"older factor $F_i$ of $F$ that has $p_H(F)=p_H(F_i)$. By Lemma \ref{lem:dis}, we have $\overline{\Delta}_H(F)=\overline{\Delta}_H(F_i)$. Also by the openess of the stability conditions, $F_i$ is $\nu_{\alpha,0}$ and $\nu_{\alpha',1}$-stable for $\alpha$ and $\alpha'$ large enough. If $F$ has a torsion Jordan--H\"older factor $F_0$, then the torsion factors must have $\ch_2(F_0)\geq0$. Since other Jordan--H\"older factors must have $0>\rk(F_i)\geq \rk(F)$ as we assume that $F\in \mathrm{Coh}^{\beta,H}(S')$, we must have some factor $F_i$ such that $\frac{H\ch_1(F)}{H^2\ch_0(F)}=\frac{H\ch_1(F_i)}{H^2\ch_0(F_i)}$ and $\frac{\ch_2(F)}{H^2\ch_0(F)}\leq\frac{\ch_2(F_i)}{H^2\ch_0(F_i)}$. Again, by the openness of the stability conditions, $F_i$ is $\nu_{\alpha,0}$ and $\nu_{\alpha',1}$-stable for $\alpha$ and $\alpha'$ large enough. Thus we can assume that $F$ is $\nu_{\alpha,0,H}$-tilt stable for all $\alpha>0$, and similarly for $\nu_{\alpha',1,H}$.

\textbf{Inequalities for stable objects:}   To apply the last lemma, we consider the line passing through $p_H(F):=(a,b)$ and $p_H(F(-2H)[1]):=(a-2b+2,b-2)$. This line has the equation 
$$(b-1)Y-X=-a+b^2-b.$$
In the proper region of $(a,b)$, it will intersect at $(\alpha_0,0)$ and $(\alpha_1,-1)$ such that $\alpha_0>0$ and $\alpha_1>\frac{1}{2}$. This is equivalent to saying that $a>b^2-b$ and $a>b^2-\frac{1}{2}$. We see that the bound above satisfies these two conditions. Thus we can apply the last lemma for $F$ and we have $\rk(F|_C)=\rk(F)$ and $\deg(F|_C)=2H\ch_1(F).$

Without loss of generality, we may assume that $\frac{H\ch_1(F)}{H^2\rk(F)}\leq\frac{1}{2}$, otherwise we can take $F^*(H)$, where $*$ denote the dual. Then we have
\begin{align*}
\chi(O_{S'},F)&=\ch_2(F)-\ch_1(F)+20\ch_0(F)\\
&\leq \hom(O_{S'},F)+\hom(O_{S'},F[2])\\
&=\hom(O_{S'},F)+\hom(O_{S'},F^*(2H))\\
&\leq \hom(O_C,F|_C)+\hom(O_C,F^*(2H)|_C)=:\Lambda.  
\end{align*}

Note that in this case, $\mu(F|_C)=\frac{2H\ch_1(F)}{\ch_0(F)}=\frac{32H\ch_1(F)}{H^2\ch_0(F)}\in(0,16].$

From now on, we let $r:=\ch_0(F)$, $\mu=32\frac{H\ch_1(F)}{H^2\ch_0(F)}$ and $d=\mu r$. As in the last section, we have several cases:

\textbf{Case 1} $\frac{H\ch_1(F)}{H^2\ch_0(F)}\in\left(0,\frac{\sqrt{69}-8}{5}\right].$
In this region, we have
$$\Lambda\leq \frac{64r^2}{64r-d}+(64-\mu)r-46r=\frac{64r^2}{64r-d}+18r-\mu r.$$

Therefore
$$\frac{\ch_2(F)}{H^2\ch_0(F)}\leq-\frac{H\ch_1(F)}{H^2\ch_0(F)}+\frac{4}{64-32\frac{H\ch_1(F)}{H^2\ch_0(F)}}-\frac{1}{8}.$$
This satisfies the bound of the proposition.

\textbf{Case 2} $\frac{H\ch_1(F)}{H^2\ch_0(F)}\in\left(\frac{\sqrt{69}-8}{5},\frac{8}{3}-\frac{\sqrt{61}}{3}\right]$. In this region, we have
\begin{align*}
\Lambda&\leq \frac{64r^2}{64r-d}+\frac{5}{1024}(64-\mu)^2r+5r-\frac{1}{8}(64-\mu)r\\
&=\frac{64r^2}{64r-d}+17r-\frac{1}{2}\mu r+\frac{5\mu^2r}{1024}.   
\end{align*}

This implies
$$\frac{\ch_2(F)}{H^2\ch_0(F)}\leq\frac{5}{16}\left(\frac{H\ch_1(F)}{H^2\ch_0(F)}\right)^2+\frac{4}{64-32\frac{H\ch_1(F)}{H^2\ch_0(F)}}-\frac{3}{16}.$$
This satisfies the bound of the proposition.

\textbf{Case 3} $\frac{H\ch_1(F)}{H^2\ch_0(F)}\in\left(\frac{8}{3}-\frac{\sqrt{61}}{3},\frac{1}{2}\right]$. In this region, we are back to the general case.
\begin{align*}
\Lambda&\leq r+\frac{5d^2}{1024r}+\frac{5}{1024}(64-\mu)^2r+5r-\frac{1}{8}(64-\mu)r\\
&=\frac{5}{512}\frac{32\times32\times(H\ch_1(F))^2}{H^2\ch_0(F)H^2}+18r-16\times\frac{H\ch_1(F)}{H^2\ch_0(F)}\times \ch_0(F)\\
&=\frac{5}{8}\frac{(H\ch_1(F))^2}{H^2\ch_0(F)}+18\ch_0(F)-16\frac{H\ch_1(F)}{H^2\ch_0(F)}\ch_0(F).    
\end{align*}

This implies
$$\frac{\ch_2(F)}{H^2\ch_0(F)}\leq\frac{5}{8}\left(\frac{H\ch_1(F)}{H^2\ch_0(F)}\right)^2-\frac{1}{8}.$$
This satisfies the bound of the proposition.

Therefore, if $\frac{H\ch_1(F)}{H^2\rk(F)}\in(0,1)$, then we have the bound of the proposition.
\end{proof}

\begin{corollary}
Let $F$ be a torsion-free $\mu_H$-slope stable sheaf on $S'$. Then the numerical  character of $F$ satisfies the bound of Proposition \ref{prop:main5}.
\end{corollary}

\begin{proof}
This is because if $F$ is $\mu_H$-slope stable, then it is $\nu_{\alpha,0,H}$-stable for $\alpha$ large enough.
\end{proof}

\begin{corollary}
The bound in Proposition \ref{prop:main5} is also true for $X_{2,4}$, where we replace $\ch_2(F)$, $H\ch_1(F)$ and $H^2\rk(F)$ by $H\ch_2(F)$, $H^2\ch_1(F)$ and $H^3\rk(F)$.
\end{corollary}

We get the following bound, which is a little weaker but easier for calculation.

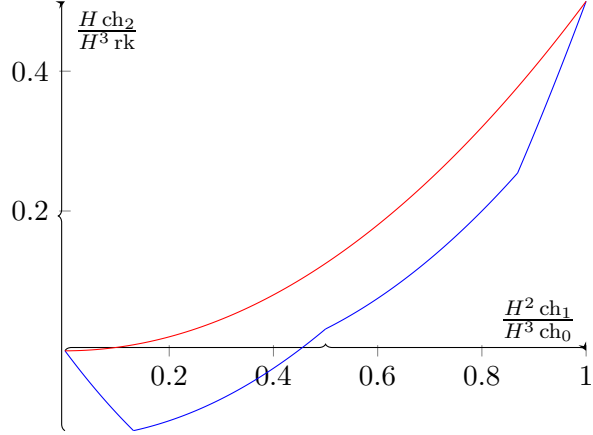
\begin{figure}
\centering
\begin{tikzpicture}
[
    dot/.style={
        draw=black,
        fill=blue!90,
        circle,
        minimum size=3pt,
        inner sep=0pt,
        solid,
    },
]
\begin{axis}[
    axis lines = center,
    xlabel = \(\frac{H^2\ch_1}{H^3\ch_0}\),
    ylabel = {\(\frac{H\ch_2}{H^3\rk}\)},
]

\addplot [
    domain=0:4/3-sqrt{13}/3, 
    samples=100, 
    color=blue,
    name path=gamma
    ]
    {x^2  - x};
\addplot [
    domain=4/3-sqrt{13}/3:0.5, 
    samples=100, 
    color=blue,
    name path=gammap1
    ]
    {5/8*x^2  - 1/8};
\addplot [
    domain=0.5:sqrt{13}/3-1/3, 
    samples=100, 
    color=blue,
    ]
    {5/8*x^2 -0.25*x};
\addplot [
    domain=sqrt{13}/3-1/3:1, 
    samples=100, 
    color=blue,
    name path=gammap2
    ]
    {x^2-0.5};

\addplot [
    domain=0:1, 
    samples=100, 
    color=red
    ]
    {0.5*x^2};

\end{axis}
\end{tikzpicture}
\caption{The Bogomolov--Gieseker type inequality (blue) and the classical Bogomolov inequality (red)} \label{graph:BG}
\end{figure}

\begin{theorem}\label{thm:main5}
On $X_{2,4}$, let $F$ be a slope semistable sheaf in $\mathrm{Coh}(X_{2,4})$ (or $\nu_{\alpha,0,H}$-tilt semistable object for $\alpha>0$, or Brill--Noether semistable). Suppose $\frac{H^2\ch_1(F)}{H^3\rk(F)}\in[-1,1]$. Then
\begin{equation}
  \frac{H\ch_2(F)}{H^3\rk(F)}\leq
  \begin{cases}
    -\frac{1}{2}\abs{\frac{H^2\ch_1(F)}{H^3\rk(F)}} & \text{if } \abs{\frac{H^2\ch_1(F)}{H^3\rk(F)}}\in[0,\frac{1}{5}]\\
    \frac{7}{16}\abs{\frac{H^2\ch_1(F)}{H^3\rk(F)}}-\frac{3}{16} & \text{if }\abs{\frac{H^2\ch_1(F)}{H^3\rk(F)}}\in[\frac{1}{5},\frac{1}{2}]\\
    \frac{9}{16}\abs{\frac{H^2\ch_1(F)}{H^3\rk(F)}}-\frac{1}{4} & \text{if }\abs{\frac{H^2\ch_1(F)}{H^3\rk(F)}}\in[\frac{1}{2},\frac{4}{5}]\\
    \frac{51}{44}\abs{\frac{H^2\ch_1(F)}{H^3\rk(F)}}-\frac{8}{11} & \text{if } \abs{\frac{H^2\ch_1(F)}{H^3\rk(F)}}\in[\frac{4}{5},\frac{10}{11}]\\
    \frac{21}{11}\abs{\frac{H^2\ch_1(F)}{H^3\rk(F)}}-\frac{31}{22} & \text{if }\abs{\frac{H^2\ch_1(F)}{H^3\rk(F)}}\in[\frac{10}{11},1]
  \end{cases}
\end{equation}
The equality can only hold when $\abs{\frac{H^2\ch_1(F)}{H^3\rk(F)}}\in\{0,\frac{1}{5},\frac{1}{2},\frac{4}{5},\frac{1}{4},\frac{10}{11},1\}$. Moreover, when $\abs{\frac{H^2\ch_1(F)}{H^3\rk(F)}}\in[0,\frac{1}{5}]$, $\frac{H\ch_2(F)}{H^3\rk(F)}\leq\left(\frac{H^2\ch_1(F)}{H^3\rk(F)}\right)^2-\abs{\frac{H^2\ch_1(F)}{H^3\rk(F)}}$ holds; when $\abs{\frac{H^2\ch_1(F)}{H^3\rk(F)}}\in[\frac{1}{5},\frac{1}{4}]$, $\frac{H\ch_2(F)}{H^3\rk(F)}\leq\frac{9}{32}\abs{\frac{H^2\ch_1(F)}{H^3\rk(F)}}-\frac{5}{32}$ holds; when $\abs{\frac{H^2\ch_1(F)}{H^3\rk(F)}}\in[\frac{1}{5},\frac{1}{2}]$, $\frac{H\ch_2(F)}{H^3\rk(F)}\leq\frac{5}{8}\left(\frac{H^2\ch_1(F)}{H^3\rk(F)}\right)^2-\frac{1}{8}$ holds; and when $\abs{\frac{H^2\ch_1(F)}{H^3\rk(F)}}\in[\frac{\sqrt{13}-1}{3},1]$, $\frac{H\ch_2(F)}{H^3\rk(F)}\leq\left(\frac{H^2\ch_1(F)}{H^3\rk(F)}\right)^2-\frac{1}{2}$ holds.
\end{theorem}

\begin{remark}
As in \cite{li2019stability}, the method in this paper is expected to apply to other cases. A possible scheme is the following: One can start with a projective Calabi--Yau threefold $X$ with ample divisor $H$, and consider a generic member $Y\in|2H|$ (or higher multiple of $H$), and a generic curve $C\in|2H_Y|$ (and still possible for a higher multiple of $H_Y$). The essence is to have a good Bogomolov--Gieseker inequality for $Y$. To do this, one can embed $C$ inside a K3 surface of Picard rank 1 (the method here) or a del Pezzo surface (as in \cite{li2019stability, koseki2020stability}), or more generally any surface for which a good $\Gamma$-curve is known. To get a good $\Gamma$-curve for a surface, one can repeatedly use the method in this paper and in \cite{li2019stability}. In particular, we expect that one can prove the Bogomolov--Gieseker inequality for $X_{3,3}\subset\mathbb{P}^5$, the complete interesection of two generic cubics, by first restricting to the surface $S_{2,3,3}$, and then using the link provided by the curve $C_{2,2,3,3}$ to the Bogomolov--Gieseker inequality of $S_{2,2,3}$. Applying the method with the curve $S_{2,2,3}\supset C_{2,2,2,3}\subset S_{2,2,2}$, we get the Bogomolov--Gieseker inequality for $S_{2,2,3}$. Like the last remark in \cite{li2019stability}, each deformation type needs a lot of calculation. 
\end{remark}

\paragraph{Conflicts of Interest}
None.

\paragraph{Acknowledgments}
The author would like to thank Dr. Chunyi Li for proposing this problem, helpful discussions and proofreading, Dr. Naoki Koseki for suggestive advice, and Mr. Jes\'us David Fern\'andez Caballero, Ms. Yefei Ma and Mr. Pedro N\'u\~nez for proofreading, and the referees for helpful comments and references.

\paragraph{Funding statement}
The author was supported by grants from the CSC-Warwick scholarship and Royal Society URF$\backslash$R1$\backslash$201129 ``Stability condition and application in algebraic geometry”.

\bibliographystyle{alpha}

\bibliography{main.bib}
\Addresses
\end{document}